\documentclass[12pt]{article}
\usepackage{amsmath,amssymb,amsthm,bbm,srcltx,graphicx,a4wide,xr,xargs, mathtools,enumerate}
\usepackage[shortlabels]{enumitem}

\usepackage[usenames,dvipsnames]{color}
\usepackage[colorlinks]{hyperref}
\hypersetup{citecolor=ForestGreen}

\usepackage{cleveref}
\numberwithin{equation}{section}

\newcommand{\toi}{\to\infty}
\newcommand{\eind}{\stackrel{d}{=}}

\newcommand{\dto}{%
	\mathrel{\vbox{\offinterlineskip\ialign{%
				\hfil##\hfil\cr
				$\scriptstyle d$\cr
				$\longrightarrow$\cr
			}}}}
\newcommand{\wto}{%
	\mathrel{\vbox{\offinterlineskip\ialign{%
				\hfil##\hfil\cr
				$\scriptstyle w$\cr
				$\longrightarrow$\cr
			}}}}
\newcommand{\vto}{%
	\mathrel{\vbox{\offinterlineskip\ialign{%
				\hfil##\hfil\cr
				$\scriptstyle v$\cr
				$\longrightarrow$\cr
			}}}}

\newcommand{\pr}{\mathbb{P}}
\renewcommand{\pr}{\mathbf{P}}                   
\newcommand{\ex}{\mathbb{E}}
\renewcommand{\ex}{\mathbf{E}}

\newcommand{\one}[1]{\boldsymbol{1}_{\{#1\}}}
\newcommand\ind[1]{\boldsymbol{1}{\left\{#1\right\}}}
\newcommand{\bone}{\boldsymbol{1}}

\newcommand{\N}{\mathbb{N}}
\newcommand{\R}{\mathbb{R}} 
\newcommand{\Z}{\mathbb{Z}}

\newcommand{\E}{\mathbb{E}}

\newcommand{\bB}{\boldsymbol{B}}
\renewcommand{\bB}{X}
\newcommand{\bN}{\boldsymbol{N}} 
\newcommand{\bQ}{\boldsymbol{Q}}
\renewcommand{\bQ}{Q}

\newcommand{\bM}{\boldsymbol{M}}

\newcommand{\bi}{\boldsymbol{i}}

\newcommand{\sA}{\mathcal{A}}

\newcommand{\sY}{\mathcal{Y}}
\newcommand{\sN}{\mathcal{N}}
\newcommand{\sF}{\mathcal{F}}
\newcommand{\tsN}{\tilde{\mathcal{N}}}

\newcommand{\eps}{\varepsilon}
\renewcommand{\epsilon}{\varepsilon}
\newcommand{\vep}{\varepsilon}
\renewcommand{\phi}{\varphi}
\renewcommand{\emptyset}{\varnothing}

\newcommand{\ball}{B}
\newcommand{\n}{\tau}
\newcommand{\dx}{\mathrm{d}}

\newcommand{\tX}{\tilde{X}}
\newcommand{\Tu}[1][u]{T_{r(#1),#1}}
\newcommand{\Tb}[1][\xi]{T_{#1^\beta,#1}}
\newcommand{\metric}{{\mathsf{m}}}

\newcommand{\tN}{\tilde{N}}
\newcommand{\tP}{\tilde{P}}
\newcommand{\tPhi}{\widetilde{\Phi}}

\newtheorem{theorem}{Theorem}[section]
\newtheorem{lemma}[theorem]{Lemma}

\newtheorem{proposition}[theorem]{Proposition}
\newtheorem{hypothesis}[theorem]{Assumption}

\theoremstyle{definition}
\newtheorem{definition}[theorem]{Definition}

\theoremstyle{remark}
\newtheorem{remark}[theorem]{Remark}
\newtheorem{example}[theorem]{Example}

\newcommand{\dint}{\mathrm{d}}

\newcommand{\shift}[1]{\varphi_{#1}}
\newcommand{\borel}{\mathcal{B}}
\usepackage{csquotes}

\newcommand{\tsum}{\textstyle\sum}
\newcommand{\tprod}{\textstyle\prod}
\newcommand{\tmax}{\textstyle\max}

\DeclareMathOperator{\Leb}{Leb}

\newcommand{\rev}[1]{\Blue{#1}}
\renewcommand{\rev}[1]{#1}

\title{Extremal behavior of stationary marked point processes}
\author{Bojan Basrak\thanks{Department of Mathematics, Faculty of
    Science, University of Zagreb, Bijeni\v{c}ka 30, 10000 Zagreb,
    Croatia; bbasrak@math.hr}, Ilya Molchanov\thanks{Institute of
    Mathematical Statistics and Actuarial Science, University of Bern,
    Alpeneggstr. 22, 3012 Bern, Switzerland;
    ilya.molchanov@stat.unibe.ch}, Hrvoje
  Planini\'c\thanks{Department of Mathematics, Faculty of Science,
    University of Zagreb, Bijeni\v{c}ka 30, 10000 Zagreb, Croatia;
    planinic@math.hr (corresponding author)}} \date{\today}

\begin{document}

\maketitle

\begin{abstract}
  We consider stationary configurations of points in Euclidean space
  which are marked by positive random variables called scores. The
  scores are allowed to depend on the relative positions of other
  points and outside sources of randomness. Such models have been
  thoroughly studied in stochastic geometry, e.g.\ in the context of
  random tessellations or random geometric graphs.

  It turns out that in a neighbourhood of a point with an extreme
  score one can often rescale positions and scores of nearby points to
  obtain a limiting point process, which we call the tail
  configuration. Under some assumptions on dependence between scores,
  this local limit determines the global asymptotics for extreme
  scores within increasing windows in $\R^d$. The main result
  establishes the convergence of rescaled positions and clusters of
  high scores to a Poisson cluster process, quantifying the idea of
  the Poisson clumping heuristic by D.~Aldous (in the point process
  setting). \rev{In contrast to the existing results, our framework allows for explicit calculation of essentially all extremal quantities related to the limiting behavior of extremes.}

  We apply our results to models based on (marked) Poisson processes
  where the scores depend on the distance to the $k$th nearest
  neighbor and where scores are allowed to propagate through a random
  network of points depending on their locations. 

\vspace{0.1cm}
  \noindent
  \textbf{Keywords:} cluster; extreme score; marked point process;
  moving maxima; nearest neighbor; Poisson convergence

  \noindent
  \textbf{MSC 2020:} Primary 60G70; Secondary 60D05; 60G55; 60G57
\end{abstract}



\section{Introduction}
\label{sec:introduction}

In studies of extreme, or rare, features in a point process
configuration, Poisson limits and extreme value distributions
naturally appear, see, for instance, Owada
\cite{owada18:_limit_betti}, Otto \cite{otto20:_poiss_poiss} or
Bobrowski et al.~\cite{bobrowski21:_poiss_palm}.  For geometric
structures without a significant local dependence Poisson limits are
not unexpected. In this case, high scores essentially arrive as
isolated points which makes analysis relatively simple. In more
complicated cases, one can try to remove clumps of large scores in the
language of Aldous~\cite{aldous:1989}.  However, if one simply
substitutes a clump (or a cluster) of large scores by a single score,
one looses important information about the geometric properties of the
limiting extremal objects. Here we derive and present a mathematically
rigorous solution to this problem.

In particular, we present techniques suitable to analyze extremes of
marked point processes in Euclidean space. Intuitively, these
processes can be viewed as a sequence of scores (i.e., random values)
recorded at random locations. We concentrate on points whose scores
are high with the aim to understand the appearance of other such
points nearby, meaning that we allow extreme scores arising in
clusters. Such clustering phenomenon has been well described by Aldous
\cite{aldous:1989}, see Chenavier and Robert \cite{chen-rob18} for a
recent analysis of extremes in random tessellations.

Our key concept is that of the tail configuration, which is closely
related to the tail process introduced in \cite{basrak:segers:2009}
and since then widely used in time series
\cite{dombry:hashorva:soulier:2018,janssen:2019,kulik:soulier:2020}. For
stationary time series this process appears after conditioning upon a
high score at zero. When adapting this idea to spatial setting,
one needs to work with Palm versions of marked point processes, which
necessarily has a point located at the origin, and then condition upon
the fact that the score at 0 is high. Passing to the limit involves
scaling down the scores, but often also requires scaling of the
locations.

The resulting tail configuration provides a local description of the
clustering phenomenon related to high scores. Looking at high scores
globally in a big domain shows that they build islands (clusters of
points) in space. Our main result Theorem~\ref{thm:main_result}
provides a limit theorem showing convergence of such global cluster
process to a marked Poisson process in the space, whose marks are
point processes themselves considered equivalent up to translation. In
this way we factor out positions of the clusters and explicitly
describe the limiting distribution and the extremal index which
provides the information about the mean cluster size. One can compare
our result with its classical counterpart for discrete time stationary
sequences as presented in \cite[Theorem~3.6]{basrak18} or
\cite[Theorem~6.1.4]{kulik:soulier:2020}.

Although our theory applies much more generally, we illustrate our
results on two examples with scores obtained from a background Poisson
point process. In the first one, the scores are simply reciprocals to
the distance between a point and its $k$th nearest neighbor. As a
special case of this example, for $k=1$ one describes the limiting
structure of the process of points with large inradii in the
Poisson--Voronoi tessellation, studied in \cite{chen-rob18} in
dimension 2.  In our second example, the points are initially marked
by i.i.d.\ random values, but the actual score at a point $t$ say,
depends also on the (weighted) values at the points in a possibly
random neighborhood of $t$. The example can be seen as a
generalization of the moving maxima model from time series
analysis. But here, we are particularly interested to see how large
values propagate in such a random network.

The paper is organized as follows.  Section~\ref{sec:main-definition}
sets up our basic definitions.  Section~\ref{sec:tail-configuration}
introduces the tail process and its spectral counterpart. Its central
result is a certain invariance property of the tail configuration
(Theorem~\ref{prop:exceedance-s}), which is related to the time-change
formula from time series, see \cite{basrak:segers:2009}; cf.\ also
\cite{planinic:2021} and \cite{last:2021} for a discussion on the
connection to standard Palm theory.
Section~\ref{sec:poiss-appr-clust} contains the main result which
provides a Poisson approximation of extremal clusters. The
construction involves the standard idea of splitting the space into
rectangular blocks. Three main assumptions consist of conditions on
block sizes and dependence within and between extremal blocks, see
Section~\ref{sec:assumpt-crefthm:m}. In
Section~\ref{sub:alternative_repr} we discuss the representation of
the extremal index and the distribution of the typical cluster. The
key concept here is the idea of an anchoring function, which is
motivated by a similar concept from random fields indexed over $\Z^d$,
see \cite{basrak:planinic:2020}. The proof of the main result is
postponed to
Section~\ref{subs:proof_main_result}. Section~\ref{sub:small_dist_kth_NN}
treats in detail the case of scores derived from the neighboring
structure of a stationary Poisson process in Euclidean space.  In
Section~\ref{sub:moving_maxima_tail} we deal with the moving maxima
model. In both examples the main steps consist of determining the tail
configuration and then checking the appropriate dependence conditions.

\rev{Consider first a very simple motivating example.
\begin{example}\label{exmp:simple}
  Let $ P = \sum \delta_{t}$ denote a homogeneous Poisson process on
  $\R^d$ independently marked by i.i.d. points $(h_t,\vep_t,\zeta_t)$
  in $\R^d \times \{0,1\} \times \R_+$.  In particular $\vep_t$ are
  i.i.d. Bernoulli random variables, with success probability $p$,
  say. Assume for simplicity that all three components of the mark are
  independent, that $h_t$'s have a symmetric continuous distribution
  around the origin with bounded support and $\zeta_t$'s have a
  regularly varying distribution with index $\alpha>0$, i.e.
  $\pr (\zeta > u) = L(u) u ^{-\alpha}$ for some slowly varying
  function $L$.  Consider now the following simple Poisson cluster
  process
  \[
    X =  \sum \delta_{(t,\zeta_t)} + \vep_t \delta_{(t+h_t,\zeta_t)}.
  \]
  Thus, at each point $t$ of the background Poisson process $P$, in
  $X$ we observe a score $\zeta_t$, which is then with probability $p$
  repeated at a shifted location $t+h_t$. Suppose now that we observe
  the process $X$ on a hypercube $[0,\n]^d$ for $\n\toi$.  Note that
  one can always find a function $ a_\n$ so that
  $\n^d \pr(\zeta>a_\n \epsilon) \to \epsilon^{-\alpha}$ for any
  $\epsilon >0$.  It is unsurprising that one gets a nontrivial limit,
  after rescaling the locations and scores of points of $X$. Indeed,
  this is immediate in the case $p=0$, in which case there are no
  clusters in $X$ and therefore (in the vague topology as explained
  below)
  \[
    T_{a_\n} X :=   \sum \delta_{(t/\n,\zeta_t/a_\n)}
    \dto N  = \sum_{i=1}^{\infty} \delta_{(U_i\,,\Gamma_i^{-1/\alpha})},
  \]
  where $\{(U_i,\Gamma_i),i\geq1\}$ are points of the Poisson process
  on $[0,1]^d\times \R_+$ with the intensity measure being the product
  of the Lebesgue measure on $[0,1]^d$ and the Lebesgue measure on
  $\R_+$. If $p=1/2$, a similar result holds, however, the large
  scores in $X$ come in clusters of size 1 or 2, so that the limit
  becomes a compound Poisson process. Without further
  adjustments, the clusters will collapse to a single location in the
  limit. Below we explain how one can prove a version of this limiting
  result which also preserves the shape the cluster in the limit under
  relatively general assumptions.
\end{example}
}

\section{Random counting measures and marked point processes}
\label{sec:point-processes}

\subsection{Basic definitions}
\label{sec:main-definition}

Consider the space $\E:=\R^d\times(0,\infty)$ with the standard
product topology and the Borel $\sigma$-algebra $\borel(\E)$. A point
in $\E$ is written as $(t,s)$, where $t\in\R^d$ is said to be the
\emph{position} and $s>0$ is the mark or \emph{score} (at $t$). 

A Borel measure $\mu$ on $\E$ is said to be a \emph{counting measure}
if it takes values $\{0,1,2,\ldots\}\cup\{\infty\}$ on $\borel(\E)$.
Denote by $\mu'$ the \emph{projection} of $\mu$ on $\R^d$, that is,
$\mu'(B)=\mu(B\times(0,\infty))$ for each Borel $B$ in $\R^d$.  We
call a counting measure \emph{simple} if its projection on $\R^d$ is
simple, that is, $\mu(\{t\}\times(0,\infty)) \leq 1$ for all
$t\in\R^d$. We write $(t,s)\in\mu$ and $\mu(t)=\rev{\mu(\{t\})=}s$ if
$\mu(\{(t,s)\})=1$; {for convenience, if
  $\mu(\{t\}\times (0,\infty))=0$ we will sometimes write $\mu(t)=0$}.
Each simple counting measure $\mu$ is uniquely represented by the set
of its atoms. To emphasize this, we write
\begin{displaymath}
  \mu=\tsum_{(t,s)\in\mu} \delta_{(t,s)},
\end{displaymath}
or equivalently, $\mu=\{(t_i,s_i):i=1,\dots,k\}$, where $k$ is the
total number of atoms in $\mu$ (which may be infinite).  For a Borel
function $f:\E\to\R$, denote
\begin{displaymath}
  \mu(f)\equiv \int f\dint\mu :=\tsum_{(t,s)\in\mu} f(t,s). 
\end{displaymath}

In the following, it is essential to single out families of sets where
counting measures take finite values. Introduce 
subfamilies $\borel_{11}, \borel_{10}, \borel_{01} \subseteq \borel(\E)$:
\begin{enumerate}[i)]
\item $A\in \borel_{11}$ if $A\subseteq B\times (\eps,\infty)$ for
  some $B\subseteq \R^d$ bounded and $\eps>0$;
\item $A\in \borel_{10}$ if $A\subseteq B\times (0,\infty)$ for some
  $B\subseteq \R^d$ bounded;
\item $A\in \borel_{01}$ if $A\subseteq \R^d\times (\eps,\infty)$ for
  some $\eps>0$.
\end{enumerate}
For consistency, we sometimes write $\borel_{00}:=\borel(\E)$.  These
families provide examples of a \emph{boundedness} or \emph{bornology}
on $\E$, see \cite{basrak:planinic:2019}, and clearly satisfy
$\borel_{11}=\borel_{10}\cap\borel_{01}$. 

Let $\sN_{ij}$ denote the family of simple counting measures with
finite values on $\borel_{ij}$, $i,j\in\{0,1\}$.  For convenience, in
the sequel denote $\sN:=\sN_{11}$. Note that
$\sN_{00} \subset \sN_{01}\cap \sN_{10}$ and
$\sN_{01}\cup \sN_{10} \subset \sN_{11} = \sN$.  The families
$\sN_{ij}$ are equipped with the \emph{vague topology} determined by
the choice of the boundedness, see \cite{basrak:planinic:2019}.

\begin{definition}
  Counting measures $(\mu_n)_{n\in \N}$ from {$\sN_{ij}$} with
  $i,j\in \{0,1\}$ are said to converge to {$\mu\in \sN_{ij}$} as
  $n\toi$ in $\borel_{ij}$ or $\borel_{ij}$-vaguely (notation
  $\mu_n\vto \mu$) if $\mu_n(f)\to\mu(f)$ as $n\toi$ for all
  continuous bounded functions $f:\E\to\R$ whose support is in some
  $B\in\borel_{ij}$.
\end{definition}



The notion of $\borel_{ij}$-vague convergence on $\sN_{ij}$ can be
seen as convergence with respect to the smallest topology on
$\sN_{ij}$ which makes the mappings $\mu\mapsto\mu(f)$ continuous for
all continuous bounded functions $f:\E\to\R$ whose support is in some
$B\in\borel_{ij}$; call this topology the \textit{$\borel_{ij}$-vague
  topology}. Since the extension of this topology to the larger space
of all Borel measures on $\E$ which are finite on elements of
$\borel_{ij}$, is known to be Polish (see \cite[Theorem
4.2]{kallenberg:2017} and \cite[Theorem 3.1]{basrak:planinic:2019}),
the $\borel_{ij}$-vague topology on $\sN_{ij}$ is separable and
metrizable. We have used the phrase \textit{$\borel_{ij}$-vague} instead of simply
\textit{vague} since one can consider $\sN_{ij}$ with respect to the
(weaker) $\borel_{i'j'}$-vague topology whenever
$\sN_{ij}\subseteq \sN_{i'j'}$ which is equivalent to
$\borel_{i'j'}\subseteq \borel_{ij}$.

\rev{In the sequel, the choice of a particular vague topology will depend 
on what kind of points in $\E$ we want to control. For example, in the definition of the tail configuration below, we will use the $\borel_{11}$-vague topology since we want to control extremal
scores located in a bounded neighborhood of a typical extremal score which is assumed to be at the origin.}
  
Define the \emph{shift operators} $\shift{z}$, $z\in \R^d$, on $\E$ by
letting $ \shift{z}(t,s) := (t-z,s)$,
and let 
\begin{align*}
  \shift{z} \mu := \tsum_{(t,s)\in \mu} \delta_{\shift{z}(t,s)}
  = \tsum_{(t,s)\in \mu} \delta_{(t-z,s)}.
\end{align*}
Thus, if $\mu$ at $z$ has score $s$, then $s$ becomes the score of
$\shift{z} \mu$ at $0$, so that the shift applies only to positions,
leaving the scores unchanged.  Since the families $\borel_{ij}$ are
invariant under shifts, the families $\sN_{ij}$ are also
invariant. Observe that the mapping $(z,\mu)\mapsto \shift{z}\mu$ from
$\R^d\times \sN_{ij}$ to $\sN_{ij}$ is continuous if $\sN_{ij}$ is
equipped with the vague topology generated by any
$\borel_{i'j'}\subseteq\borel_{ij}$.

The sets $\sN_{ij}$, $i,j=0,1$, are equipped with the Borel
$\sigma$-algebra generated by the maps $B\mapsto\mu(B)$ for all
$B\in\borel(\E)$; this coincides with the Borel $\sigma$-algebra
generated by $\borel_{ij}$-vaguely open sets.  A \emph{random counting
  measure} is a random element $X$ in $\sN$. It is called
\emph{stationary} if $\shift{z}X$ and $X$ coincide in distribution for
all $z\in\R^d$.

If a random counting measure $X$ takes values from a smaller family
$\sN_{10}$, then $X$ is called a \emph{marked point process} on $\R^d$
(with marks in $(0,\infty)$).  Then, for all bounded $A\subset\R^d$,
we have $X(A\times(0,\infty))<\infty$ a.s., that is, the number of
$t\in A$ such that $(t,s)\in X$ for some $s>0$, is almost surely
finite. We assume throughout that this number has a finite mean. If $X$ is
stationary, the expected value of $X(A\times(0,\infty))$ is
proportional to the Lebesgue measure of $A\in\borel(\R^d)$. The
coefficient of proportionality $\lambda$ is said to be the intensity
of $X$. Later on we usually assume that $\lambda=1$.

Each stationary marked point process $X$ on $\R^d$ of finite intensity
admits its \emph{Palm version} $\tX$, which is a marked point process
on $\R^d$, satisfying \rev{refined Campbell's theorem}
\begin{align}
  \label{eq:refined_Campbell}
  \ex\left[\tsum_{(t,s)\in X} h(t,\shift{t}X)\right]
  = \lambda \int_{\R^d} \ex\left[h(t,\tX)\right] \dint t 
\end{align}
for all measurable $h:\R^d\times \sN_{10}\to \R_+$.  The Palm version
$\tX$ has the following invariance property
\begin{align}
  \label{eq:mecke}
  \ex\left[\tsum_{(t,s)\in \tX} h(-t,\shift{t}\tX)\right]
  =\ex\left[\tsum_{(t,s)\in \tX} h(t,\tX)\right] 
\end{align}
for all measurable $h:\R^d\times \sN_{10} \to \R_+$, see
\cite[Theorem~13.2.VIII]{daley:verejones:2008}.
Note that $\tX$ almost surely contains the point $(0,\xi)$; the random
variable $\xi:=\tX(0)$ is said to be the score of the Palm version at
the origin.

Let $\wto$ denote weak convergence of probability measures and $\dto$
the corresponding convergence in distribution. Distributional
convergence of random counting measures in $\sN$ is understood with
respect to a particular version of the vague topology, and so relies
on the choice of the corresponding boundedness. It is well known that
$X_n\dto X$ in $\borel_{ij}$ if and only if the Laplace functionals of
$X_n$
\begin{displaymath}
  L_{f}(X_n):=\ex\ \exp \left\{ - \tsum_{(t,s)\in X_n} f(t,s) \right\}
\end{displaymath}
converge to $L_f(X)$ as $n\toi$ for all continuous
$f : \E\to [0,\infty)$ with support in $\borel_{ij}$, see \cite[Theorem 4.11]{kallenberg:2017}.

\subsection{A general construction of scores}
\label{sec:gener-constr-scor}

In our main examples we deal with marked point processes $X$ derived from
a marked Poisson point process using the following general
construction.  Let $P$ be an independently marked stationary Poisson
process in $\R^d$, where $\R^d$ is the space of locations and the
marks take values from $(0,\infty)$. Note that trivial amendments make
it possible to consider the marks taking values in a
general Polish space, and allow this construction to be applied to a
general marked point processes. The intensity measure of $P$ is the
product of the Lebesgue measure on $\R^d$ (possibly, scaled by a
constant) and a probability measure $m$ on $(0,\infty)$.

Consider a measurable function
$\psi: \R^d \times {\sN_{10}} \to (0,\infty)$ such that
$\psi(t-z,\varphi_z \mu) = \psi(t,\mu)$ for all $z\in\R^d$. In the
following $\psi$ is called a \emph{scoring function}. For
$ \mu := \tsum \delta_{(t,z)} \in {\sN_{10}}$ denote
\begin{displaymath}
  \Psi(\mu) := \tsum_{(t,z)\in\mu} \delta_{(t, \psi(t,\mu))}\,.
\end{displaymath}
This defines a mapping from {$\sN_{10}$ to $\sN_{10}$}. While this
mapping does not change locations of points, it equips each point
$t\in\mu'$ with a new score $s=\psi(t,\mu)$. The same construction can
be clearly applied to a Poisson process in $\R^d$ without marks, which
fits in the above framework by letting all the marks equal to 0 say.

The shift-invariance property of $\psi$ implies that
\begin{equation}\label{eq:PsiShift}
  \Psi(\varphi_z\mu)  = \varphi_z(\Psi(\mu))\,.
\end{equation}

By the Poisson assumption, the Palm version of $P$ is given by
\begin{displaymath}
  \tP=P+\delta_{(0,\zeta)},
\end{displaymath}
where $\zeta$ has distribution $m$, and is independent of $P$. 

\begin{lemma}
  \label{lemma:Psi}
  The Palm version of $\Psi(P)$ is given by $\Psi(\tP)$.
\end{lemma}
\begin{proof}
  The Palm version $\tP$ satisfies
  \begin{align*}
    \ex \Big[ \tsum_{(t,u) \in P} h(t,\varphi_t P ) \Big]
    &  =  \lambda \int_{\R^d} \ex  [h(t,\tP)]\dx t \,.
  \end{align*}
  Therefore, \eqref{eq:PsiShift} yields that
  \begin{equation*}
    \ex \Big[ \tsum_{(t,u) \in P} h(t,\varphi_t \Psi(P) ) \Big]
    =	\ex \Big[ \tsum_{(t,u) \in P} h(t,\Psi \varphi_t  (P) ) \Big] 
    =  \lambda \int_{\R^d} \ex  [h(t, \Psi(\tP))] \dx t\,.
  \end{equation*}
  By the definition of Palm measure, the left-hand side is
  $$
  \lambda \int_{\R^d} \ex  \big[h(t,\widetilde{\Psi(P)})\big]
  \dx t\,,
  $$
  so that the Palm version of $\Psi(P)$ is indeed $\Psi(\tP)$.
\end{proof}

\section{Tail configuration}
\label{sec:tail-configuration}

Define a family of scaling operators $T_{v,u}:\E\to\E$, $u,v>0$, by
\begin{align}\label{eq:scaling0}
  T_{v,u}(t,s):=(t/v,s/u), \quad (t,s)\in\E.
\end{align}
For every $\mu\in\sN$, define its scaled version $T_{v,u}\mu$ by
letting $(T_{v,u}\mu)(A):=\mu(T_{v^{-1},u^{-1}}A)$ for all Borel $A$.
Equivalently, $T_{v,u}\mu = \tsum \delta_{(t/v,s/u)}$ if 
$\mu = \tsum \delta_{(t,s)}$, \rev{meaning that the atoms of
  $T_{v,u}\mu$ are obtained by applying the transformation $T_{v,u}$
  to the atoms of $\mu$.}

In the following we mostly work with counting measures scaled by $\Tu$
for $u>0$, where a function $r:(0,\infty)\to (0,\infty)$, is fixed and
regularly varying at infinity, i.e.,
\begin{align}
  \label{eq:rv_scaling}
  r(u)=l(u)u^\beta ,\quad  u>0 \, ,
\end{align}
for some $\beta\in \R$ and a slowly varying function $l$. We will
refer to $r$ and $\beta$ as the \emph{scaling function} and
\emph{scaling index}, respectively. Note that $\beta=0$ and
$r(u)\equiv 1$ are allowed.  

\begin{definition}\label{def:tail_conf}
  Fix a function $r$ which is regularly varying at infinity.  Let
  $X\in\sN_{10}$ be a stationary marked point process on $\E$ with
  Palm version $\tX$ and the score at the origin being $\xi$. If there
  exists a random counting measure $Y\in\sN$ such that
  $Y(\{0\} \times (1,\infty))=1$ a.s.\ and
  \begin{align}\label{eq:tail_process}
    \pr(\Tu \tX \in \cdot \mid \xi>u) \wto \pr(Y\in \cdot\,)
    \quad \text{ as } u\toi 
  \end{align}
  with respect to the $\borel_{11}$-vague topology, then $Y$ is called
  the \emph{tail configuration} of $X$ (with respect to the scaling
  function $r$).
\end{definition}
 
Note that the tail configuration $Y$ is assumed to be simple and it
necessarily contains the point $(0,\eta)$ with $\pr(\eta > 1)=1$. We call
the random variable $\eta$ the \emph{tail score} at the origin.  While
$X$ and $\tX$ are marked point processes and thus belong to $\sN_{10}$, the
tail configuration $Y$ in general takes values in $\sN$, which is a
larger family.

\rev{
\begin{example}[continuation of Example \ref{exmp:simple}]
  For the stationary point process $X$ of our initial example, it is
  straightforward to see that the Palm version of $\tX$ is
  $\tX \eind X + C_0$, where $C_0$ is the Palm version of the typical
  cluster of $X$ independent of $X$, see Chiu et
  al. \cite[Section~5.3]{chiu:2013}. More precisely,
  $C_0 \eind \delta_{(0,\xi)} + I \delta_{(h_0,\xi)}$. Here $I$
  represents a Bernoulli random variable independent of the random
  pair $(h_0,\xi)$, which has the same distribution as any of the
  pairs $(h_t,\zeta_t)$. However, due to the size biasing phenomenon, $I$ has different distribution from $\varepsilon_t$: in the
  case $p=1/2$, for instance, $\pr(I=1) = 2/3$. Observe that, due to the
  regular variation assumption, the distributions
  $\pr(\xi /u \in \cdot \mid \xi>u)$ converge to a Pareto distribution
  with parameter $\alpha>0$.  Moreover, the
  independence of $X$ and $C_0$  implies that
  $T_{1,u} X \dto 0$ even if we condition on $\xi>u$, therefore
  for $r(u) \equiv 1$
  \begin{displaymath}
    \pr(\Tu \tX \in \cdot \mid \xi>u)
    = \pr(T_{1,u} C_0 \in \cdot \mid \xi>u) + o(1)
    \wto \pr(Y\in \cdot\,)\quad \text{as}\; u\toi,
  \end{displaymath}
  where
  \[ 
    Y  \eind \delta_{(0,\eta)} + I \delta_{(h_0,\eta)}\,,
  \]
  and where $h_0$ is independent of $\eta$ which has the Pareto
  distribution with index $\alpha$.
\end{example}
}

It is convenient to denote
\begin{equation}
  \label{eq:1}
  \sA_y:=\big\{\mu\in\sN :(0,s)\in\mu \;\text{for some}\; s>y\big\}.
\end{equation}
Observe that the tail process $Y$ almost surely belongs to $\sA_1$. \rev{For $c>0$, $x\in \R^d$, by $B_c(x)$ we denote the open Euclidean ball around $x$  of radius $c$, and set $B_c:=B_c(0)$.}

\begin{proposition}\label{prop:S0t_is_Pareto}
  If \eqref{eq:tail_process} holds then the score at the origin $\xi$
  has a regularly varying tail, that is,
  \begin{align}
    \label{eq:S_0_is_RV}
    \lim_{u\toi}\frac{\pr(\xi>uy)}{\pr(\xi>u)}=y^{-\alpha},
    \quad y>0 \, ,
  \end{align}
  for some $\alpha>0$, and the tail score at the origin
  $\eta$ is $\mathrm{Pareto}(\alpha)$-distributed, that is,
  \begin{align}\label{eq:S0t_is_Pareto}
    \pr(\eta>y)=y^{-\alpha}, \quad y\geq 1 \, .
  \end{align}
\end{proposition}
\begin{proof} 
  \rev{Assume that $\mu,\mu_1,\mu_2,\dots $ are counting measures in
  $\sA_1$ such that $\mu_n \vto \mu$ in $\borel_{11}$. Since
  $\mu \in \sA_1$, one can always find a bounded set
  $B\in \borel_{11}$ of the form
  $B= \ball_\vep \times (1+\vep,\infty)$ such that $\mu(B) = 1$ (i.e.\
  $(0,\mu(0))$ is the only point of $\mu$ in $B$) and
  $\mu (\partial B) = 0$. Convergence $\mu_n \vto \mu$ implies
  (see, e.g., \cite[Proposition 2.8]{basrak:planinic:2019}) that
  $\mu_n(0) \to \mu(0)$.} In other words, the score at the origin is a
  continuous function on $\sN \cap \sA_1$.  By a continuous mapping
  argument, $\xi/u = \Tu \tX (0)$, conditionally on $\xi >u$,
  converges in distribution to $\eta = Y(0)$ as $u \toi$. More
  precisely,
  \begin{displaymath}
    \frac{\pr(\xi>uy,\xi>u)}
    {\pr(\xi>u)}\to \pr(\eta>y) \quad \text{as }\; u\toi
  \end{displaymath}
  for all $y> 0$ which are continuity points for $\eta$.  Standard
  arguments now yield that (\ref{eq:S_0_is_RV}) holds for some
  $\alpha>0$ \rev{(derived by analysing the tail of $\xi$)}, and
  (\ref{eq:S0t_is_Pareto}) follows immediately.
\end{proof}

The constant $\alpha>0$ from \eqref{eq:S0t_is_Pareto} will be called
the \emph{tail index} of $X$, $\xi$ and $\eta$. \rev{Observe that scaling
index of the function $r$ from \eqref{eq:rv_scaling} is not related to
$\alpha$.} Further, the point process
\begin{align}
  \label{eq:spectral_tail_conf}
  \Theta:=\Tb[\eta] Y
\end{align}
in $\sN$ is called the \emph{spectral tail configuration} of $X$.  By
definition, $\Theta$ almost surely contains the point $(0,1)$.

\begin{proposition}\label{prop:spectral_is_indep}
  The spectral tail configuration is independent of $\eta$, and
  satisfies
  \begin{align}\label{eq:conv_to_spectral}
    \pr(\Tu[\xi] \tX \in \cdot \mid \xi>u)
    \wto \pr(\Theta\in \cdot) \quad\text{ as }\; u\toi \, .
  \end{align} 
\end{proposition}
\begin{proof}
  Recall the set $\sA_y$ from \eqref{eq:1} and note that $\sA_1$
  consists of all $\mu\in\sN$ such that
  $\mu(\{0\}\times(1,\infty))\geq1$. Consider the family of mappings
  $H_u$, $u>0$, defined by $H_u(\mu):=(s,r(su)/r(u), \mu)$, where
  $s=\mu(0)$ is the score at the origin of $\mu\in\sA_1$.

  Let $\mu_u\vto \mu$ (in $\borel_{11}$) as $u\toi$  for some $\mu_u$
  and $\mu$ from 
  $\sA_1$. Denote $s_u:=\mu_u(0)$ and $s:=\mu(0)$. Then $s_u\to s$ as
  $u\toi$.  Since $r$ is regularly varying with index $\beta$, the
  convergence $r(yu)/r(u)\to y^{\beta}$ as $u\toi$ holds locally
  uniformly in $y$ on $(0,\infty)$, see
  e.g.~\cite[Proposition~2.4]{resnick:2007}, so that
  $r(s_u u)/r(u)\to s^{\beta}$ as $u\toi$.  Therefore,
  $H_u(\mu_u)\to (s, s^{\beta}, \mu)$ as $u\toi$. The extended
  continuous mapping theorem applied to (\ref{eq:tail_process}) (see
  \cite[Theorem~5.5]{billingsley:1968}) yields that
  \begin{align*}
    \pr((\xi/u, r(\xi)/r(u), \Tu\tX)
    \in \cdot \mid \xi>u)\wto
    \pr((\eta, \eta^{\beta}, Y)\in \cdot \,) 
  \end{align*}
  on $(1,\infty)\times (0,\infty)\times \sN$. Another application
  of the continuous mapping theorem yields
  \begin{align*}
    \pr((\xi/u,\Tu[\xi]\tX)\in \cdot \mid \xi>u)
    \wto \pr((\eta, \Theta)\in \cdot \,)
  \end{align*}
  on $(1,\infty)\times \sN$. This yields
  (\ref{eq:conv_to_spectral}) and, together with (\ref{eq:S_0_is_RV}),
  \begin{align*}
    \pr(\eta>y, \Theta \in B )
    &= \lim_{u\toi}  \pr(\xi>uy, \Tu[\xi]\tX\in B \mid \xi>u) \\
    &= \lim_{u\toi}\frac{\pr(\xi>uy)}{\pr(\xi>u)}
      \; \pr(\Tu[\xi]\tX \in B\mid \xi>uy) \\
    &= y^{-\alpha} \pr(\Theta \in B) = \pr(\eta>y)\pr(\Theta \in B) 
  \end{align*}
  for all $y\ge 1$ and all Borel subsets $B\subseteq \sN$ such
  that $\pr(\Theta\in \partial B)=0$. This implies that $\eta$ and
  $\Theta$ are independent, since the class of all such $B$'s (denoted
  by $\mathcal{S}$) is closed under finite intersections and generates
  the Borel $\sigma$-algebra on $\sN$. The latter fact follows,
  since the vague topology on $\sN$ is separable and metrizable,
  so one can represent every open subset of $\sN$ as a countable
  union of open balls which are elements of $\mathcal{S}$.
\end{proof}
 
To conclude this section, we show that the invariance property
\eqref{eq:mecke} of the Palm distribution $\tX$ induces a similar
property of the tail configuration $Y$ which, as in
\cite[Section~2]{planinic:2021}, could be called
\emph{exceedance-stationarity}, cf.\ also \cite{last:2021}.

\begin{theorem}
  \label{prop:exceedance-s}
  For every  measurable $h:\R^d\times \sN \to [0,\infty)$,
  \begin{align}\label{eq:exc_stat}
    \ex\left[\tsum_{(t,s)\in Y} h(t,Y)\one{s>1}\right]
    =\ex\left[\tsum_{(t,s)\in Y} h(-t,\shift{t}Y)
    \one{s>1}\right]\, .
  \end{align}
\end{theorem}
\begin{proof}
  Since $\Tu$ scales the scores with $u^{-1}$, a score in $\Tu\tX$
  exceeding 1 corresponds to a score in $\tX$ exceeding $u$.
  Thus, by \eqref{eq:mecke},
  \begin{align*}
    \ex\Big[\tsum_{(t,s)\in \Tu\tX} & h(t,\Tu\tX)\one{s>1}
    \one{\xi>u}\Big]\\
    &=\ex\left[\tsum_{(t,s)\in \tX} h(t/r(u),\Tu\tX)
      \one{\tX(t)>u}
      \one{(\Tu\tX)(0)>1}\right]\\
    &=\ex\left[\tsum_{(t,s)\in \tX} h(-t/r(u),\Tu\shift{t}\tX)
      \one{(\shift{t}\tX)(-t)>u}
      \one{(\Tu\shift{t}\tX)(0)>1}\right]\\
    &=\ex\left[\tsum_{(t,s)\in \tX} h(-t/r(u),\shift{t/r(u)}\Tu\tX)
      \one{\tX(0)>u}
      \one{\tX(t/r(u))>u}\right]\\
    &=\ex\left[\tsum_{(t,s)\in \Tu\tX} h(-t,\shift{t}\Tu\tX)\one{\xi>u}
      \one{s>1}\right].
  \end{align*}
  We aim to show that both sides (if normalized by $\pr(\xi>u)$)
  converge to the corresponding sides of \eqref{eq:exc_stat}. However,
  a direct application of \eqref{eq:tail_process} is not possible
  since \rev{ the functionals
  \begin{align*}
    \mu &\mapsto \tsum_{(t,s)\in \mu}  h(t,\mu)\one{s>1} \, , \\
    \mu &\mapsto \tsum_{(t,s)\in \mu} h(-t,\shift{t}\mu)
      \one{s>1}    
  \end{align*}
  for $\mu\in \sA_1$ (see (\ref{eq:1})), are not bounded, even if $h$
  is bounded.  }
  
\rev{Fix a bounded continuous function $h:\R^d\times\sN\to[0,\infty)$
  such that, for some $c>0$, $h(t,\mu)=0$ for all $t\notin B_c$.
  Furthermore, fix $k\in\N$, $a>2c$, and consider the maps
  $H_1,H_2:\sA_{1}\to [0,\infty)$ given by
  \begin{align*}
    H_1(\mu)&:=\tsum_{(t,s)\in\mu} h(t,\mu)
    \one{\mu(\ball_a\times(1,\infty))\leq k}\one{s>1},\\
    H_2(\mu)&:=\tsum_{(t,s)\in\mu} h(-t,\shift{t}\mu)
    \one{(\shift{t}\mu)(\ball_a\times(1,\infty))\leq k}\one{s>1}.
  \end{align*}
  Both maps are bounded by $k\sup h$ since $H_1(\mu)=H_2(\mu)=0$
  whenever $\mu(B_c\times (1,\infty))>k$. Moreover, we claim that
  $H_1$ and $H_2$ are continuous on all $\mu \in \sA_1$ such that
  \begin{itemize}
  \item[(i)] $\|t-x\|\neq a$ for all $(t,s), (x,v) \in \mu$, and
  \item[(ii)]  $s\neq 1$ for all $(t,s)\in \mu$.
  \end{itemize}
  Denote by $C_a$ the set of all such $\mu$'s. If $a>2c$, the
  indicators $\one{(\shift{t}\mu)(\ball_a\times(1,\infty))\leq k}$,
  $t\in B_c$, depend on points of $\mu$ in $B_{2a}\times
  (1,\infty)$. Since for each $\mu \in C_a$ one can find $\eps>0$ such
  that $\mu(\partial (B_{2a+\eps} \times (1,\infty)))=0$, and since
  $B_{2a+\eps} \times (1,\infty)$ is in $\borel_{11}$, properties of
  $\borel_{11}$-vague convergence (see
  \cite[Proposition~2.8]{basrak:planinic:2019}) imply that maps $H_1$
  and $H_2$ are continuous on $C_a$.

  Since $Y$ has at most countably many points (as any other random
  element of $\sN$), it is easy to show that for all but at most
  countably many $a>0$, it almost surely holds that $\|t-x\|\neq a$
  for all $(t,s), (x,v) \in Y$.  Furthermore, since
  \begin{displaymath}
    Y=\sum_{(t,\theta)\in \Theta} \delta_{(t \eta^{\beta}, \theta
      \eta)} \, , 
  \end{displaymath} 
  where $\eta$ has a non-atomic distribution and is independent of the
  spectral tail configuration $\Theta$, it immediately follows that
  with probability zero $Y$ contains $(t,s)$ with $s=1$, equivalently,
  $\Theta$ contains $(t,\theta)$ with $\theta=\eta^{-1}$. 
  Thus, one can find $a>2c$ such that $\pr(Y\in C_a)=1$.
  
  Since $\ex [H_1(\Tu\tX) \mid \xi >u]= \ex [H_2(\Tu\tX) \mid \xi >u]$
  for all $u>0$ (argue exactly as in the beginning of the proof),
  applying \eqref{eq:tail_process} we obtain that
  $\ex H_1(Y)=\ex H_2(Y)$. Letting $k\to\infty$ yields
  \eqref{eq:exc_stat} for all nonnegative continuous bounded functions
  $h$ which vanish for $t\notin \ball_c$. We can further remove the latter restriction by letting $c\toi$. We claim that this ensures that (\ref{eq:exc_stat}) holds for all nonnegative measurable functions.
  
  Observe first that, when viewed as functions of $h$, both sides of
  (\ref{eq:exc_stat}) define a Borel measure on $\R^d \times \sN$ --
  denote them by $\nu_1$ and $\nu_2$. Since these two measures
  coincide on all nonnegative bounded continuous functions, then they
  coincide on the $\pi$-system of all open subsets of $\R^d \times
  \sN$. Measures $\nu_1$ and $\nu_2$ are in general unbounded, but
  they take finite values on the sets of the form $B_c\times C_{a,k}$
  where $C_{a,k}=\{\mu \in \sN : \mu(B_a \times (1,\infty)) \leq k\}$
  for $c>0, a>2c,  k\in \N$. While the set $C_{a,k}$ is not open in the $\borel_{11}$-vague topology, its subset
  \begin{align*}
      C_{a,k}':=\{\mu \in C_{a,k} : \mu(\partial(B_a \times (1,\infty))=0\}
  \end{align*}
  is open, which implies that 
  \begin{align*}
      \nu_1(\cdot \,  \cap (B_c \times C_{a,k}')) =  \nu_2(\cdot \,  \cap (B_c \times C_{a,k}')) \;\; \; \text{on $(\R^d\times \sN, \borel(\R^d\times \sN))$} \, , 
  \end{align*}
  for all $c>0, a>2c,  k\in \N$. As already explained above, for every fixed $c>0$ one can find $a=a(c)>2c$ such that $\pr(Y(\partial(B_a \times (1,\infty))=0)=1$. With this choice of $a$, $\nu_1$ and $\nu_2$ put zero mass on $C_{a.k}\setminus C_{a,k}'$, for all $k\in \N$, so in particular
  \begin{align*}
      \nu_1(\cdot \,  \cap (B_c \times C_{a,k})) =  \nu_2(\cdot \,  \cap (B_c \times C_{a,k})) \;\; \; \text{on $(\R^d\times \sN, \borel(\R^d\times \sN))$} \, , 
  \end{align*}
  for all $k\in \N$. Since $\mu(B_c\times (1,\infty))<\infty$ for all $\mu \in \sN$, we have that $ C_{a,k} \uparrow \sN$ as $k\toi$. 
  Thus, by letting $k\toi$, and then $c\toi$ we obtain that $\nu_1$ and $\nu_2$ coincide on $(\R^d\times \sN, \borel(\R^d\times \sN))$, which proves the claim.}
\end{proof}

\begin{remark}
  Exceedance-stationarity property (\ref{eq:exc_stat}) and
  the polar decomposition from \Cref{prop:spectral_is_indep} yield
  \begin{align*}
    \ex\left[\tsum_{(y,s)\in \Theta}
    h(-y/s^{\beta},\Tu[s]\shift{y}\Theta)\right]
    =\ex\left[\tsum_{(y,s)\in\Theta} h(y,\Theta)s^{\alpha}\right] 
  \end{align*}
  for every measurable $h:\R^d\times\sN \to [0,\infty)$, see
  \cite[Remark~2.11]{planinic:2021}. This property of the spectral
  tail configuration can be seen as the analogue of the
  \emph{time-change formula} known to characterize the class of all
  spectral tail processes (and thus tail processes) of regularly
  varying time series, see
  \cite{dombry:hashorva:soulier:2018,janssen:2019}.
\end{remark}

\section{Poisson approximation for extremal clusters}
\label{sec:poiss-appr-clust}

In what follows assume that $X$ is a stationary marked point process
on $\R^d$ of unit intensity with marks (scores) in $(0,\infty)$, which
admits a tail configuration $Y$ in the sense of \Cref{def:tail_conf}.
The main goal of this section is to describe the limiting behavior of
scores of $X$ in {$[0,\n]^d$} which exceed a suitably chosen high
threshold, as $\n$ and the threshold size tend to infinity.

\subsection{Extremal blocks}
\label{sec:extremal-blocks}

Let $(a_\n)_{\n>0}$ be a family of positive real numbers chosen such that
\begin{align}
  \label{eq:a_n}
  \lim_{\n\toi} \ex\Big[\tsum_{(t,s)\in X, t\in [0,\n]^d}
  \one{s>a_\n}\Big]=\lim_{\n\toi} \n^d \pr(\xi>a_\n)= 1 \, ,
\end{align} 
where the first equality follows from refined Campbell's theorem
(\ref{eq:refined_Campbell}). 
By (\ref{eq:S_0_is_RV}),
\begin{align}\label{eq:a_n_eps}
  \lim_{\n\toi} \n^d \pr(\xi>a_\n \epsilon)
  = \epsilon^{-\alpha}, \quad \epsilon>0 \, .
\end{align}

Let $(b_\n)_{\n>0}$ be a family of positive real numbers such that
$b_\n/\n\to 0$ as $\n\toi$. Divide the hypercube $[0,\n]^d$ into
blocks of side length $b_\n$ defined as
\begin{align}\label{eq:blocks_of_indices}
  J_{\n,\bi}:= \bigtimes_{j=1}^d [(i_j-1)b_\n, i_j b_\n]
  \subset \R^d
\end{align}
for $\bi=(i_1,\dots,i_d) \in I_{\n}:= \{1,\dots, k_\n\}^d$, where
$k_\n := \big \lfloor \n/b_\n\big\rfloor$.\footnote{Technically speaking, we are dividing the hypercube $[0, k_{\n}b_{\n}]^d$. However, in applications this edge effect is easily seen to be negligible.} For every
$\bi \in I_{\n}$, define
\begin{equation}\label{eq:extr-block}
  \bB_{\n,\bi}:=  X_{J_{\n,\bi}},
\end{equation}
which is the restriction of $X$ to $J_{\n,\bi}\times(0,\infty)$. 

For fixed $\tau$ and $\varepsilon$, think of \emph{clusters} of
extremal scores of $X$ as blocks $\bB_{\n,\bi}$ which contain at least
one score exceeding $a_\n \epsilon$. For every $\bi\in I_\tau$, by
(\ref{eq:a_n_eps}) and since $b_{\n}/\n \to 0$,
\begin{align*}
  \pr(\bB_{\n,\bi} \text{ is an extremal cluster})
  &=\pr(\tmax_{(t,s)\in \bB_{\n,\bi}} s>a_\n \epsilon)\\
  &\leq b_{\n}^d \pr(\xi>a_\n \epsilon) \to 0 \quad \text{as}\; \n\toi ,
\end{align*}
where \rev{we bound the probability with expectation, then used refined
Campbell's theorem (\ref{eq:refined_Campbell}) and, finally, \eqref{eq:a_n}.}

\subsection{Space for extremal blocks}
\label{sub_l_tilde}

Recall that $X$ is a random element of the space $\sN_{10}$ so that
the blocks $\bB_{\n, \bi}$ can be considered as elements of \rev{$\sN_{01}$}
which consists of simple counting measures on $\E$ with finite values
on $\R^d\times(\eps,\infty)$ for all $\eps>0$. Recall further that
$\sN_{01}\subset \sN=\sN_{11}$ and that, on $\sN_{01}$,
$\borel_{01}$-topology is stronger than the $\borel_{11}$-topology. We
now define a metric $\metric$, generating the $\borel_{01}$-vague
topology on $\sN_{01}$.

Let $\mu,\nu\in \sN_{01}$ be such that $\mu(\E),\nu(\E)<\infty$ (i.e.,
$\mu,\nu\in \sN_{00}$). If $\mu(\E)\neq \nu(\E)$ set
$\metric_0(\mu, \nu)=1$, and if $\mu(\E)=\nu(\E)=k\in \N_0$ and
$\mu=\tsum_{i=1}^k \delta_{(t_i,s_i)}$,
$\nu=\tsum_{i=1}^k \delta_{(t_i',s_i')}$, define
\begin{displaymath}
  \metric_0(\mu,\nu)=\min_{\Pi} \max_{i=1,\dots,k}
  \left[(|t_i - t_{\Pi(i)}'| \vee {|s_i - s_{\Pi(i)}'|}) \wedge 1
  \right] \, ,
\end{displaymath}
where the minimum is taken over all permutations $\Pi$ of
$\{1,\dots,k\}$. Note that $\metric_0$ is a metric generating the weak
(that is, $\borel_{00}$-vague) topology on $\sN_{00}$, see
\cite[Proposition~2.3]{schuhmacher08}.
While the authors of \cite{schuhmacher08} assume that the ground space
is compact, an easy argument justifies the claim for the space
$\E=\R^d\times(0,\infty)$. Observe also that $\metric_0$ is by
construction bounded by $1$ and shift-invariant, that is,
$\metric_0(\shift{y}\mu, \shift{y}\nu)=\metric_0(\mu,\nu)$ for all
$y\in \R^d$ and $\mu,\nu\in \sN_{00}$.

For general $\mu,\nu \in \sN_{01}$, set
\begin{align}\label{eq:metric_on_l0}
  \metric(\mu,\nu) := \int_{0}^{\infty}
  \metric_0(\mu^{1/u}, \nu^{1/u}) e^{-u} \dx u\,,
\end{align}
where $\mu^{1/u}$ (and similarly $\nu^{1/u}$) is the restriction of
$\mu$ on $\R^d\times (1/u,\infty)$. One can show that this is indeed a
metric on $\sN_{01}$ which is bounded by $1$, is shift-invariant, and
that it generates the $\borel_{01}$-vague topology. The latter claim
follows since $\mu_n$ converges to $\mu$ in $\sN_{01}$ with respect to
this topology if and only if $\mu_n^{1/u}$ converges weakly to
$\mu^{1/u}$ for Lebesgue almost all $u\in(0,\infty)$, see
\cite[Proposition~A2.6.II]{daley:verejones:2003} and
\cite{morariu:2018}.

We will actually work with a \emph{quotient space} of $\sN_{01}$.  For
$\mu,\nu\in \sN_{01}$, set $\mu\sim \nu$ if $\shift{y}\mu=\nu$ for
some $y\in \R^d$, and denote by $\tilde{\sN}_{01}$ the quotient space
of shift-equivalent counting measures in $\sN_{01}$. Denote by
\begin{displaymath}
  [\mu]=\{\nu \in \sN_{01}: \mu\sim \nu\}=\{\shift{y}\mu : y\in \R^d\}
\end{displaymath}
the equivalence class of $\mu\in \sN_{01}$. Define
\begin{align*}
  \tilde{\metric}([\mu],[\nu])
  :=\inf_{y,z\in \R^d} \metric(\shift{y}\mu,\shift{z}\nu),
  \quad  \mu, \nu\in \sN_{01}\, .
\end{align*}

\begin{lemma}
  Function $\tilde{\metric}$ is a metric on $\tilde{\sN}_{01}$, and
  $(\tilde{\sN}_{01},\tilde{\metric})$ is a separable metric
  space. Moreover, $\tilde{\metric}([\mu_n], [\mu])\to 0$ as $n\toi$
  (denoted by $[\mu_n]\vto [\mu]$) for $\mu_n,\mu\in \sN_{01}$
  if and only if there exist $y_n\in \R^d$, $n\in \N$, such that
  $\shift{y_n}\mu_n \vto \mu$ in $\sN_{01}$.
\end{lemma}
\begin{proof}
  Since $\metric$ is shift-invariant,
  \begin{align}\label{eq:tilde_d_inter1}
    \tilde{\metric}([\mu],[\nu])
    = \inf_{z\in \R^d} \metric(\mu, \shift{z}\nu)
    = \inf_{z\in \R^d} \metric(\shift{z}\mu, \nu) 
  \end{align}
  for all $\mu, \nu\in \sN_{01}$. \rev{It is now easy to show that
    (\ref{eq:tilde_d_inter1}) implies that $\tilde{\metric}$ is a
    pseudo-metric on $\tilde{\sN}_{01}$, and that
    $(\tilde{\sN}_{01},\tilde{\metric})$ is separable since
    $(\sN_{01}, \metric)$ is separable. This follows by a direct
    application of \cite[Lemma~2.5.1]{thesis}; the only non-trivial
    step is to show that $\tilde{\metric}$ satisfies the triangle
    inequality.}
  
  \rev{To show that $\tilde{\metric}$ is actually a metric, assume
    that $\tilde{\metric}([\mu], [\nu])=0$ for some
    $\mu,\nu\in \sN_{01}$. By (\ref{eq:tilde_d_inter1}), there exists
    a sequence $(z_n)_{n\in\N} \subseteq \R^d$ such that
    $\metric(\shift{z_n}\mu, \nu)\to 0$. If $\nu$ is the null measure,
    then it follows easily that $\mu$ is also the null measure, so that
    $[\mu]=[\nu]$. If $\nu$ is not the null measure, for $\eps>0$
    small enough, we have that
    $1\leq \nu(\R^d \times (\eps, \infty))<\infty$. Since
    $\shift{z_n}$ only translates positions, the sequence $(z_n)_n$
    must be bounded. Indeed, otherwise one would have that for
    infinitely many $n$'s,
    $\metric_0((\shift{z_n}\mu)^{1/u}, \nu^{1/u})=1$ for
    $1/u\leq \eps$, and, in particular,
    $\metric(\shift{z_n}\mu, \nu) \geq e^{1/\eps} - 1 >0$.  Thus,
    there exists a subsequence $(z_{n_k})_{k\in \N}$ such that
    $z_{n_k} \to z \in \R^d$ as $k\toi$. Since
    $\lim_{k\toi}\metric(\shift{z_{n_k}}\mu, \shift{z}\mu)\to 0$, and
    since $\metric$ is a metric, we conclude that
    $\shift{z}\mu = \nu$, i.e.\ $[\mu]=[\nu]$.}
\end{proof}

Next, recall the scaling operator $\Tu$ defined at
(\ref{eq:scaling0}). Observe that
\begin{align*}
  \Tu \shift{y}\mu = \shift{y/r(u)} \Tu \mu 
\end{align*} 
for all $u>0$, $y\in \R^d$ and $\mu\in \sN_{01}$. Since
$[\Tu\mu]=[\Tu\nu]$ whenever $\mu\sim\nu$, the scaling operator
$\Tu[u] [\mu]:=[\Tu\mu]$ is well defined on $\tilde{\sN}_{01}$. For
$\mu\in\sN_{01}$, denote by
\begin{displaymath}
  M(\mu):=\tmax_{(t,s)\in\mu} s
\end{displaymath} 
the \emph{maximal score} of $\mu$, which is necessarily
finite. Observe that the maximal score is shift-invariant and thus it
is well defined on $\tilde{\sN}_{01}$.

Finally, for every $\mu\in \sN_{01}$ such that $M(\mu)>0$ (that is,
$\mu$ is not the null measure), define
\begin{align}
  \label{eq:first_maximum_anchor}
  A^{\mathrm{fm}}(\mu)=\min \{t: (t,s)\in \mu, s=M(\mu)\} \, ,
\end{align}
where the minimum is taken with respect to the lexicographic order on
$\R^d$. Thus, $A^{\mathrm{fm}}(\mu)$ is the position where the maximal
score of $\mu$ is attained; if there is a tie, the first position is
chosen. This is well defined since $\mu$ has at most finitely many
scores exceeding any $\varepsilon>0$.

\subsection{Main result}

For each $\n>0$ consider the point process 
\begin{align}\label{eq:PP_blocks}
  N_\n = \left\{\big(\bi b_{\n} / \n \, , \;
  \Tu[a_{\n}][\bB_{\n,\bi}]\big) : \bi \in I_{\n}\right\}
\end{align}
of (rescaled) blocks together with their ``positions''; observe that
$\bi b_{\n}$ is one of the corners of $J_{\n,\bi}$ defined at
(\ref{eq:blocks_of_indices}).  While the definition of $N_\n$ also
depends on the choice of the block size length $b_{\n}$, for
simplicity we do not include it in the notation. The exact choice of
the position component is immaterial: subsequent results hold if it is
arbitrarily chosen inside the block using any deterministic or
randomized procedure.

Let $\tsN_{01}^{*}$ be equal to the space $\tsN_{01}$, but without the
null measure. Furthermore, let $\bN$ be the space of all counting
measures on $[0,1]^d \times \tsN_{01}^{*}$
which are finite on all Borel sets
$B\subseteq [0,1]^d \times \tsN_{01}^{*}$ such that, for some
$\eps=\eps(B)>0$, $M(\mu)>\eps$ for all $(t,\mu)\in B$.  Equip $\bN$
with the vague topology generated by the same family of sets.

\begin{theorem}\label{thm:main_result}
  Let $X$ be a stationary marked point process on $\E$ which admits a
  tail configuration $Y$ in the sense of \Cref{def:tail_conf} with
  tail index $\alpha>0$ and a scaling function $r$ of scaling index
  $\beta\in\R$. Let $\Theta$ be the corresponding spectral tail configuration
  defined in (\ref{eq:spectral_tail_conf}). Finally, fix a family of
  block side lengths $(b_{\n})_{\n>0}$.

  If Assumptions~\ref{hypo:t_n}, \ref{hypo:AC} and
  \ref{hypo:AI_of_extr_blocks} below hold, then
  $\pr(Y \in \sN_{01})= \pr(\Theta \in \sN_{01})=1$ and
  \begin{align}\label{eq:pp_convergence}
    N_{\n} \dto N
    := \sum_{i=1}^{\infty} \delta_{(U_i\, , \,
    T_{\Gamma_i^{\beta/\alpha},\Gamma_i^{1/\alpha}} [\bQ_i])}  \quad \text{ as } \n \toi
  \end{align}
  in $\bN$, where $\{(U_i,\Gamma_i,Q_i),i\geq1\}$ are points of the
  Poisson process on $[0,1]^d\times \R_+\times \sN_{01}$ with the
  intensity measure being the product of the Lebesgue measure on
  $[0,1]^d$, the Lebesgue measure on $\R_+$ scaled by
  \begin{align}\label{eq:extremal_index}
    \vartheta:=\pr(A^{\mathrm{fm}}(\Theta)=0)\in (0,1] \,,
  \end{align}
  and the probability distribution of a random element $\bQ$ in
  $\sN_{01}$ given by
  \begin{align}
    \label{eq:normalized_typical_cluster_cond}
    \pr(\bQ\in \cdot \,)
    = \pr(\Theta \in \cdot \mid A^{\mathrm{fm}}(\Theta)=0) \,.
  \end{align}
\end{theorem}

The proof of \Cref{thm:main_result} is presented in
\Cref{subs:proof_main_result}. The mark component in limiting point
process $N$ from \eqref{eq:pp_convergence} can be viewed as the
scaling transformation of the equivalence classes $[Q_i]$ by
multiplying the scores with $y_i:=\Gamma_i^{-1/\alpha}$ and positions
with $y_i^\beta$. Note that $\{y_i,i\geq1\}$ form a Poisson process on
$(0,\infty)$ with intensity ${\vartheta} \alpha y^{\alpha-1}\dx y$.
Furthermore, note that the point process $\bQ$ with distribution
(\ref{eq:normalized_typical_cluster_cond}) necessarily satisfies
$M(Q)=1$ almost surely.

If $\bB_{\n, (j)}$ and $\bi_{(j)} b_{\n}$, $j=1,2,\dots, k_{\n}^d$,
denote the original blocks and their positions, relabelled so that
\begin{align*}
  M(\bB_{\n, (1)}) \geq M(\bB_{\n, (2)}) \geq \dots \geq M(\bB_{\n,
  (k_{\n}^d)})  \, ,  
\end{align*} 
the continuous mapping theorem applied to (\ref{eq:pp_convergence})
yields the convergence
\begin{align*}
  \Big(\bi_{(j)}b_{\n}/\tau ,
  \Tu[a_{\n}] [\bB_{\n, (j)}] \Big)_{j=1,\dots, k}
  \dto \Big(U_{i}, T_{\Gamma_i^{\beta/\alpha},
  \Gamma_i^{1/\alpha}}[\bQ_i]\Big)_{i=1,\dots, k} \quad \text{as}\; \n\toi
\end{align*}
in the space $([0,1]^d\times \tsN_{01})^{k}$ for every fixed $k\geq 1$.
In particular, for $k=1$ this (modulo some edge effects which are
easily shown to be negligible) implies that
\begin{align}\label{eq:easy_consq}
  \lim_{\n \toi}\pr\big(a_{\n}^{-1}\tmax_{(t,s)\in X, t\in [0,\n]^d}
  s  \leq y\big) = \pr(\Gamma_1^{-1/\alpha}\leq y)
  = e^{-\vartheta y^{-\alpha}} \, , \quad y>0 \, . 
\end{align} 
Thus, the limiting distribution of the rescaled maximal score of $X$
in $[0,\n]^d$ is the nonstandard Fr\'echet distribution. Since the
point process of locations in $X$ is assumed to be a unit rate
stationary process and since the marginal score satisfies
(\ref{eq:a_n_eps}), the value of $\vartheta$ deserves to be called the
\emph{extremal index} of $X$.

The second ingredient of the limiting point process in
(\ref{eq:pp_convergence}) is the distribution of $\bQ$ which can be
seen as the asymptotic distribution of a normalized \emph{typical}
cluster of exceedances of $X$. \rev{In contrast, the tail
  configuration $Y$ is not typical. Since it contains what can be intuitively understood as a uniformly selected exceedance in   $X$, the  distribution of $Y$ is biased towards clusters with more exceedances.}
In fact, the relationship between the tail configuration
and the typical cluster of exceedances of $X$ is similar to the
relationship between a stationary point process on $\R^d$ and
its Palm version, and this relationship is discussed in detail in
\cite{planinic:2021} for random fields over $\Z^d$.

\begin{figure}[h]
	\centering
	\includegraphics[scale=0.45]{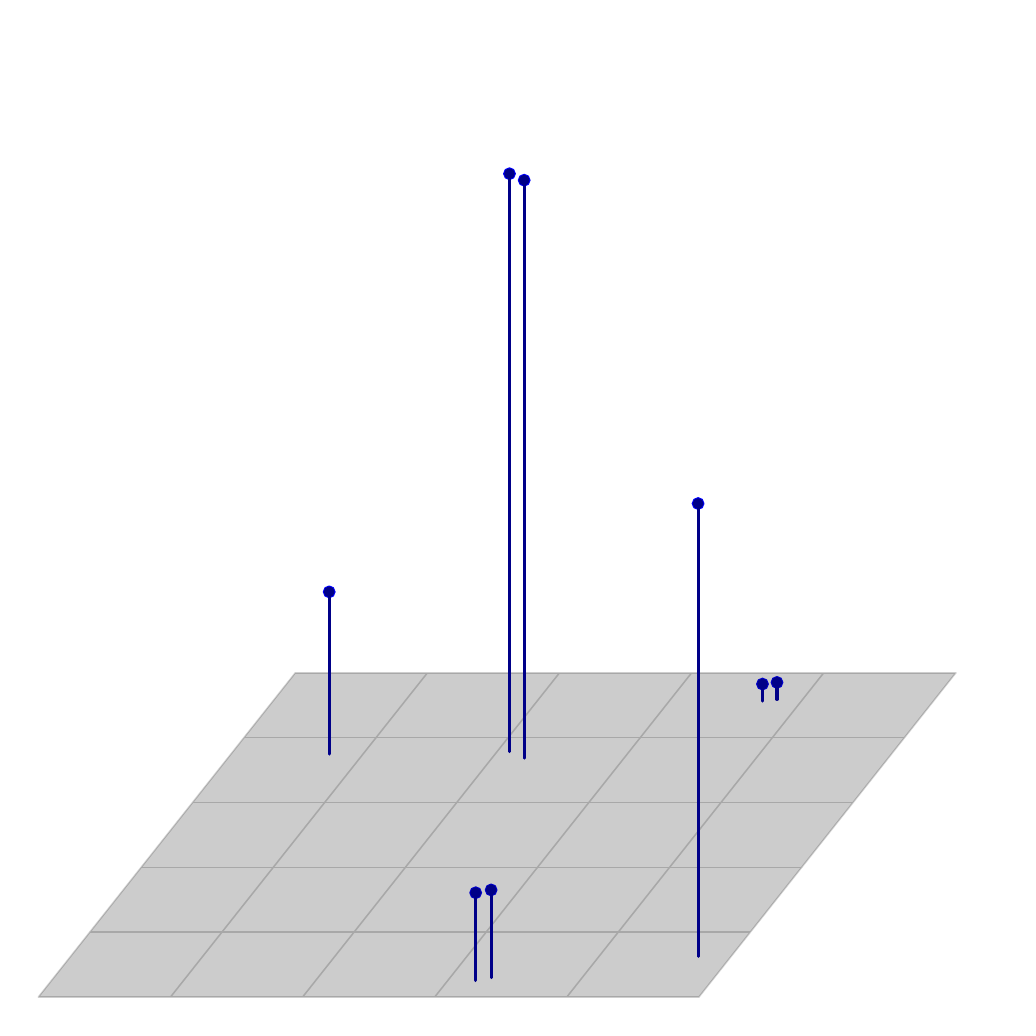}
	\includegraphics[scale=0.45]{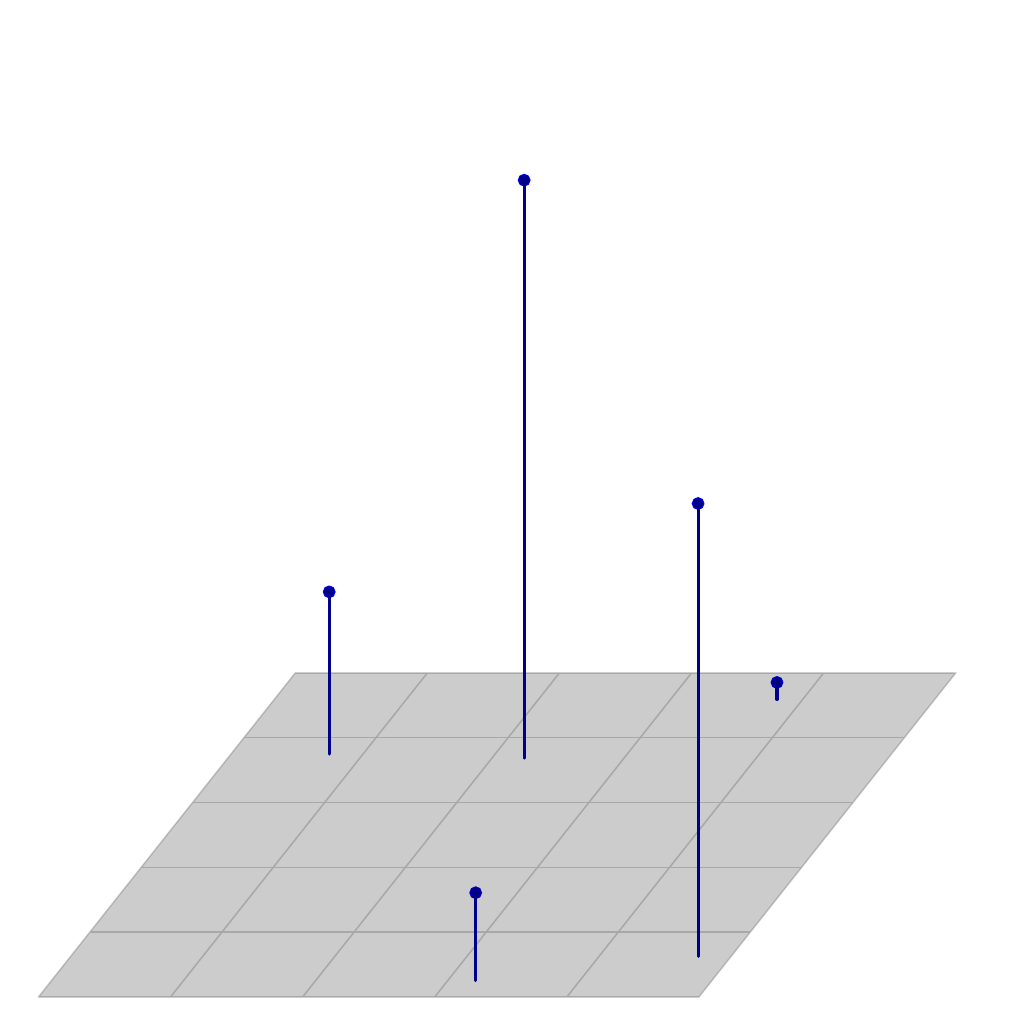}
	\caption{
	\rev{	Simulated exceedances of scores over a fixed threshold and an illustration of the block limit of $N_{\n}$ in Theorem~\ref{thm:main_result} for
		the case with  and  without clustering. The model is taken from Example~\ref{exmp:simple} with $p=1/2$ and $p=0$.} }
	\label{fig:Main1}
\end{figure}

\rev{
\begin{example}[continuation of Example~\ref{exmp:simple}]
  Recall that for the stationary point process $X$ of our initial
  example, we found the tail configuration in the form
  $Y = \delta_{(0,\eta)} + I \delta_{(h_0,\eta)}$, with the distribution
  of $I,\, h_0$ and $\eta$ described above. By
  \eqref{eq:spectral_tail_conf} the spectral tail configuration has
  the form $\Theta = \delta_{(0,1)} + I \delta_{(h_0,1)}$. For $p =1/2$ we
  have $\pr(I=1) = 2/3$, therefore,
  \begin{displaymath}
    \vartheta=\pr(A^{\mathrm{fm}}(\Theta)=0) = \pr(I=0) + \pr(I=1)
    \pr(\min(0,h_0) = 0) = 2/3.
  \end{displaymath}
  Finally, it is not difficult to show that
  $Q = \delta_{(0,1)} + \vep \delta_{(h^+_0,1)}$ where $h^+_0$ has the
  distribution of $h_0$ restricted to the points larger than 0 in the
  lexicographical order on $\R^d$ and $\vep$ is an independent
  Bernoulli random variable with $\pr(\vep=1)=1/2$.
\end{example}
}

\subsection{Assumptions of \Cref{thm:main_result}}
\label{sec:assumpt-crefthm:m}

Fix a family of positive real numbers $(b_{\n})_{\n>0}$ which
represent the block side lengths.

\begin{hypothesis}[\textbf{{on scaling}}]
  \label{hypo:t_n}
  The family $(b_{\n})_{\n>0}$ satisfies
  \begin{align}\label{eq:T_lambda}
    \lim_{\n\toi} \frac{r(a_{\n})}{b_\n} = \lim_{\n\toi} \frac{b_\n}{\n} = 0 \, .
  \end{align}
\end{hypothesis}

By regular variation of the scaling function $r$, (\ref{eq:T_lambda})
yields that $r(a_{\n}\epsilon)/b_\n\to 0$ as $\n\toi$ for all
$\epsilon>0$ as well.  If the scaling function $r$ is a constant (as
it always was for the case of random fields over $\Z^d$), then
necessarily $b_{\n}\toi$. However, if $r(a_{\n})\to 0$, one can take
$(b_{\n})_{\n}$ to be a constant or even such that $b_{\n}\to 0$.

Recall that, for $B\in \borel(\R^d)$, we denote by $\mu_B$ the
restriction of $\mu\in \sN$ to $B\times (0,\infty)$. Furthermore,
recall that $\tX$ denotes a Palm version of $X$ and $\xi=\tX(0)$.
 
\begin{hypothesis}[\textbf{{on dependence within a block/}}
  \textbf{anticlustering}]
  \label{hypo:AC}
  For all $\epsilon,\delta,c>0$, 
  \begin{align}\label{eq:AC}
    \lim_{u\toi}\limsup_{\n\toi} \pr\big(M(\tX_{C_{\n,u}})
    >a_\n \delta \mid \xi>a_\n \epsilon \big) = 0 \, , 
  \end{align}
  where
  \begin{displaymath}
    C_{\n,u}:=\ball_{b_\n c}\setminus \ball_{r(a_\n \epsilon)u},
    \quad \n,u>0 .
  \end{displaymath}
\end{hypothesis}
\rev{Assumption \ref{hypo:AC} 
concerns the maximum score of $\tX$ in the annulus $C_{\n,u} =
\ball_{b_\n c}\setminus \ball_{r(a_\n \epsilon)u}$, which, for
a sufficiently large $u$, is far away from the origin and still
relatively small compared to the size of $[0,\n]^d$. The assumption
simply states that if we condition on a large score at the origin, we are
unlikely to see another large value in such an area. Roughly speaking,
it prevents  clustering of large scores beyond a certain distance. For
a further illustration of this condition see
Example~\ref{exmp:simple4}. }

Our next and final assumption essentially requires that extremal
blocks asymptotically behave as if they were independent. To state it,
we introduce some additional notation.  Let $\sF$ be the
family of all shift-invariant measurable functions
$f:\sN_{01} \to [0,\infty)$ such that $f(0) = 0$, and for which there
exists some $\delta>0$ such that, for all $\mu\in \sN_{01}$,
\begin{equation}
  \label{eq:fmudelta}
  f(\mu)=f(\mu^{\delta}),
\end{equation}
where $\mu^{\delta}$ denotes the restriction of $\mu$ to
$\R^d\times (\delta,\infty)$, that is, the value of $f$ depends only on
scores of $\mu$ which are larger than $\delta$.

For a family of positive real numbers $(l_{\n})_{\n}$, for every
$\n>0$ and $\bi=(i_1,\dots,i_d)\in I_{\n}$, cut off the edges of
$J_{\n, \bi}$ by $l_{\n}$, that is, consider
\begin{align}\label{eq:trimmed_blocks}
  \hat{J}_{\n,\bi}:= \bigtimes_{j=1}^d [(i_j-1)b_\n +l_{\n},
  i_j b_\n - l_{\n}]\, , 
\end{align}
and define the corresponding trimmed block
\begin{align*}
  \widehat{\bB}_{\n,\bi}:=  X_{\hat{J}_{\n,\bi}}.
\end{align*}

\begin{hypothesis}[\textbf{{on dependence between extremal blocks}}]
  \label{hypo:AI_of_extr_blocks}
  There exists a family $(l_{\n})_{\n>0}$, satisfying
  \begin{align*}
    \lim_{\n\toi}\frac{l_{\n}}{b_{\n}} = 0 \, , 
  \end{align*}
  and such that 
  \begin{multline}\label{eq:AI_blocks_cond}
    \ex\Big[\tprod_{\bi \in I_{\n}}
    \exp\left\{-f_{\n,\bi}(\Tu[a_{\n}]\widehat{\bB}_{\n,\bi})\right\}\Big] -
    {\tprod_{\bi\in I_{\n}}\ex\Big[ \exp\left\{-f_{\n,
            \bi}(\Tu[a_{\n}]\widehat{\bB}_{\n,\bi})\right\}\Big]}\to 0
  \end{multline}
  as $\n\toi$ for any family of functions $f_{\n, \bi} \in \sF$,
  $\n>0, \bi \in I_{\n}$, which satisfy \eqref{eq:fmudelta} for the same
  $\delta>0$.
\end{hypothesis}

Informally speaking, (\ref{eq:AI_blocks_cond}) holds if extremal
scores are only locally dependent. The crucial issue in
(\ref{eq:AI_blocks_cond}) is that the value of
$f_{\n, \bi}(\Tu[a_{\n}]\widehat{\bB}_{\n,\bi})$ depends only on
points $(t,s)\in \widehat{\bB}_{\n,\bi}$ with $s>a_{\n}\delta$, and
that for any $(t,s)\in \widehat{\bB}_{\n,\bi}$ and
$(t',s')\in \widehat{\bB}_{\n,\bi'}$ for $\bi\neq \bi'$, one has
$|t-t'|\geq l_{\n}$.

\rev{
\begin{example}[continuation of Example~\ref{exmp:simple}]
  \label{exmp:simple4}
  Recall that
  $ X = \sum \delta_{(t,\zeta_t)} + \vep_t \delta_{(t+h_t,\zeta_t)} $
  where i.i.d. points $(h_t,\vep_t,\zeta_t)$ have independent
  components with the first having bounded support, i.e.
  $\pr(\| h\| \leq H) = 1$ for some constant $H\geq 0$. If we set
  $r\equiv 1$, Assumption \ref{hypo:t_n} holds for any $b_\n \toi$ and
  $b_\n= o(\n)$. Assumption \ref{hypo:AC} is also easily verified
  in this context:
simply recall that $\tX = X+ C_0$ with
  $C_0 \eind \delta_{(0,\xi)} + I \delta_{(h_0,\xi)}$. Therefore, for any
  $u>H$, there is no point of $C_0$ in the annulus
  $C_{\n,u}=\ball_{b_\n c}\setminus \ball_{ u} $ and thus
  \begin{displaymath}
    \pr(M(\tX_{C_{\n,u}}) > a_\n \delta \mid \xi>a_\n \epsilon )
    = \pr( M(X_{C_{\n,u}})> a_\n \delta )
    \leq \mathrm{const} \cdot b_\n^d \pr(\zeta >
    a_\n \delta) \to 0
  \end{displaymath}
  as $\tau \toi$. Similarly, because any cluster fits in a ball of
  radius $H$, Assumption \ref{hypo:AI_of_extr_blocks} holds
  immediately, if we trim the block with $l_\n = H$.
\end{example}
}

\subsection{Alternative representations of $\vartheta$ and $Q$}
\label{sub:alternative_repr}

Assume that $Y$ is the tail configuration of a stationary marked point
process $X$ on $\E$, with tail index $\alpha>0$ and a scaling function
$r$ of scaling index $\beta$. In this subsection we give some
Palm-like properties of the tail configuration under the assumption
that $Y\in \sN_{01}$ a.s., which e.g.\ holds under the anticlustering
condition (\ref{eq:AC}), see \Cref{prop:conv_to_tail_in_N_star}.  All
of the following results are based on the exceedance-stationarity
property (\ref{eq:exc_stat}) of the tail configuration.
   
For any $\mu\in \sN_{01}$, denote
\begin{align*}
  e(\mu) := \big\{t\in \R^d : (t,s)\in \mu \text{ and } s>1\big\} \, .
\end{align*}  
By definition, $0\in e(Y)$ a.s.  A function $A:\sN_{01}\to \R^d$ will
be called an \emph{anchoring function} if, for all $\mu \in \sN_{01}$
such that $e(\mu)\neq \emptyset$,
\begin{itemize}
\item[(i)] $A(\mu)\in  e(\mu)$;
\item[(ii)] $A(\shift{z}\mu)=A(\mu)-z$ for all $z\in \R^d$
  (shift-equivariance).
\end{itemize}   
A typical example of an anchoring function is the first maximum
anchor $A^{\mathrm{fm}}$ from (\ref{eq:first_maximum_anchor}). Another
one is the first exceedance anchor
\begin{align}\label{eq:first_exc_anchor}
  A^{\mathrm{fe}}(\mu):=\min e(\mu)\, ,
\end{align}
where the minimum is taken with respect to the lexicographic order.

\begin{lemma}\label{lem:theta_positive}
  If $Y\in \sN_{01}$ a.s., then
  \begin{displaymath}
   \vartheta_A=\pr(A(Y)=0)>0
  \end{displaymath}
  for any anchoring function $A$.
\end{lemma} 
\begin{proof}
  We adapt the proof of \cite[Lemma~3.4]{basrak:planinic:2020}.
  Using (\ref{eq:exc_stat}) and the shift-equivariance property of $A$,
  \begin{align*}
    1&= \ex\left[\tsum_{(t,s)\in Y}\one{A(Y)=t, s>1}\right]
       = \ex\left[\tsum_{(t,s)\in Y} \one{A(\shift{t}Y)=-t, s>1}\right] \\
     &= \ex\left[\tsum_{(t,s)\in Y} \one{A(Y)=0, s>1}\right]
       = \ex\left[\one{A(Y)=0}\tsum_{(t,s)\in Y} \one{s>1}\right] \, .
  \end{align*}
  Note that $\pr(Y\in \sN_{01})=1$ implies
  $\tsum_{(t,s)\in Y} \one{s>1}<\infty$ almost surely. Thus,
  $\pr(A(Y)=0)>0$, since otherwise the last expression above would
  vanish.
\end{proof}

\begin{proposition}\label{prop:anchors_are_equivalent}
  If $\pr(Y\in \sN_{01})=1$, then $\vartheta_A$ and the distribution
  $\pr\big([Y] \in \cdot \mid A(Y)=0\big)$ on $\tsN_{01}$
  do not depend on the choice of the anchoring function $A$.
\end{proposition}
\begin{proof}
  This result parallels \cite[Lemma~3.5]{basrak:planinic:2020} and it
  can proved in the same manner by using \eqref{eq:exc_stat} instead
  of \cite[Property~(3.8)]{basrak:planinic:2020}.  We omit the
  details.
\end{proof}

The above results yield alternative representations for $\vartheta$ in
(\ref{eq:extremal_index}). Indeed, recall the spectral tail
configuration $\Theta=\Tb[\eta] Y$ (where $\eta=Y(0)$) and observe
that $A^{\mathrm{fm}}(\Theta)=0$ if and only if
$A^{\mathrm{fm}}(Y)=0$. Thus
$\vartheta=\vartheta_{A^{\mathrm{fm}}}
=\pr(A^{\mathrm{fm}}(\Theta)=0)$ which is further equal to
$\vartheta_A=\pr(A(Y)=0)$ for an arbitrary anchoring function
$A$. Since $Y = T_{\eta^{-\beta}, \eta^{-1}} \Theta$ with $\Theta$
independent of the $\mathrm{Pareto}(\alpha)$ random variable $\eta$,
\begin{align}\label{eq:spectral_decomp_of_Z}
  \pr\big(Y \in \cdot \mid A^{\mathrm{fm}}(Y)=0\big)
  = \int_{1}^{\infty} \pr\big(T_{u^{-\beta}, u^{-1}}
  Q \in \cdot \, \big) \alpha u^{-\alpha-1} \dx u
\end{align}
on $\sN_{01}$, where $Q$ has distribution
(\ref{eq:normalized_typical_cluster_cond}) on $\sN_{01}$. Using this
fact, one can, as in \cite[Proposition~3.9]{planinic:2021}, prove that
$Q$ from (\ref{eq:normalized_typical_cluster_cond}) satisfies
\begin{align}\label{eq:normalized_typical_cluster_lo}
  \pr ([\bQ] \in \cdot \,) = \vartheta^{-1}
  \ex \bigg[\ind{\Tb[{M(\Theta)}] [\Theta] \in \cdot \,}
  \frac{M(\Theta)^{\alpha}}{\tsum_{(t,s)\in\Theta}s^\alpha} \bigg] 
\end{align}
on $\tsN_{01}$. In particular,
\begin{align}\label{eq:extremal_index_2}
  \vartheta=\ex\bigg[\frac{M(\Theta)^{\alpha}}
  {\tsum_{(t,s)\in\Theta}s^\alpha}\bigg] \, ,
\end{align}
and $(\bQ_i)_{i\geq 1}$ in (\ref{eq:pp_convergence}) can be chosen such
that their common distribution on $\sN_{01}$ satisfies
\begin{align}\label{eq:normalized_typical_cluster}
  \pr (\bQ \in \cdot \,) = \vartheta^{-1}
  \ex \bigg[\ind{\Tb[{M(\Theta)}] \Theta \in \cdot \,}
  \frac{M(\Theta)^{\alpha}}{\tsum_{(t,s)\in\Theta}s^\alpha} \bigg] .
\end{align}

\section{Examples: tail configurations, extremal indices and typical
  clusters}
\label{sec:examples-1}

\subsection{Small distance to the $k$th nearest neighbor}
\label{sub:small_dist_kth_NN}

\rev{Below we consider the situation when each point of the stationary
  Poisson process is equipped with the score, being the reciprocal of
  the distance to its nearest neighbour. Then large scores identify
  points with small distances to nearest neighbours and the tail
  configuration describes the positions of points in a cluster of
  points which all are located near to each other. Note that large
  distances to the nearest neighbour recently studied in
  \cite{otto-chen23} indicate isolated points which do not form
  clusters.}

Fix a $k\in \N$. For a set $I\subset \R^d$ which has at most a finite
number of points in any bounded region and $t\in\R^d$, let
$\rho_k(t,I)$ denote the distance from $t$ to its $k$th nearest
neighbor in $I\setminus\{t\}$.  Note that $\rho_k(t,I)<a$ if and only
if $I\setminus\{t\}$ has at least $k$ points in the open ball of
radius $a$ centered at $t$.

Let $P$ be a homogeneous unit intensity Poisson process on $\R^d$, and
let $X$ be the marked point process obtained by attaching the score
$\rho_k(t,P)^{-1}$ to each of $t\in P$, so that the score is the
reciprocal to the distance from $t\in P$ to its $k$th nearest
neighbor. Thus, one can write $X = \Psi(P)$ with the scoring function
$\psi(s,\mu) = \rho_k(s,\mu)^{-1}$, see
Section~\ref{sec:gener-constr-scor}.  By Lemma~\ref{lemma:Psi}, the
Palm version $\tX$ is obtained by the same procedure applied to
$P+\delta_0$, so that the points in $\tX$ are located at all points
$t$ from $P+\delta_0$ and the score at $t$ is given by
$s=\rho_k(t,P+\delta_0)^{-1}$.  In particular, the score of the Palm
version at the origin is
\begin{align*}
  \xi=\rho_k(0,P+\delta_0)^{-1}. 
\end{align*}
\rev{Since the random variable $P(\ball_{1/u})$ has Poisson
  distribution with mean $C_d u^{-d}$, where $C_d$ is the volume of
  the unit ball in $\R^d$, and this mean goes to zero as $u\toi$}, it
is straightforward to see that
\begin{align*}
  \pr(\xi> u) = \pr(P(\ball_{1/u})\ge k)
  \sim \pr(P(\ball_{1/u})=k) = e^{-C_d u^{-d}}
  \frac{(C_d u^{-d})^{k}}{k!} 
  \quad \text{as }\; u\toi. 
\end{align*}
Thus, 
\begin{align}\label{eq:small_dist_kNN_marginal}
  \pr(\xi>u) \sim \frac{C_d^k}{k!} u^{-dk} \quad
  \text{as} \; u\toi \, ,
\end{align}
i.e., $\xi$ has a regularly varying tail with tail index $\alpha=dk$.

\subsubsection*{Tail configuration}

\begin{proposition}
  \label{prop:small_dist_kNN_tail_conf}
  \label{cor:small_dist_kNN_tail_conf_RV}
  For every $k\in \N$, the tail configuration of $X$ exists with
  the normalizing function $r(u)=u^{-1}$ (i.e., $\beta=-1$) and is given by
  \begin{displaymath}
    Y:= \Psi(\sY)=\{(U_i,\rho_k(U_i,\sY)^{-1}), i=0,\dots,k\},
  \end{displaymath}
  where $\sY:=\{U_0,\dots, U_k\}$, $U_0=0$, and $U_1,\dots,U_k$ are
  i.i.d.\ uniform on $\ball_1$. 
\end{proposition}

\begin{remark}
  In this case the tail configuration $Y$ is an element of
  $\sN_{10}\subset \sN$, and as the proof below shows, the convergence
  in \eqref{eq:tail_process} is valid even in the stronger
  $\borel_{10}$-vague topology.
\end{remark}

\begin{proof}[Proof of \Cref{cor:small_dist_kNN_tail_conf_RV}]
  Recall that $X = \Psi(P)$ for the scoring function
  $\psi(s,\mu) = \rho_k(s,\mu)^{-1}$, and let $\tP= \delta_0 + P$ be a
  Palm version of $P$.  Observe that the conditional distribution of
  $\tX$ given $\xi >u$ coincides with the conditional distribution of
  $\Psi(\tP)$ given $P(B_{1/u}) \geq k$.  \rev{In view of this and
    taking into account that
    $\pr(P(\ball_{1/u})\ge k) \sim \pr(P(\ball_{1/u})=k)$, it is easy
  to see that for any nonnegative $f$ the Laplace functional of $\tX$
  conditional on $\{\xi >u\}$} satisfies
  \begin{displaymath}
    L_f \big( \Tu \tX \,  \big| \, \xi>u  \big) =
    L_f \big(\Tu \Psi(\tP) \,  \big| \, P (B_{1/u}) \geq k  \big) \sim 
    L_f \big(\Tu \Psi(\tP) \,  \big| \, P (B_{1/u}) = k  \big)
  \end{displaymath}
  as $u \toi$.

  Furthermore, $\tP$ conditionally on $P(B_{1/u}) = k $ has
  the distribution of
  $$
  P^{(u)} := \delta_0 + \tsum_{i=1}^k  \delta_{{U_i}/{u}}
  + P_{B_{1/u}^c} \,,
  $$
  where $U_1,\dots,U_k$ on the right-hand side are uniformly
  distributed in $B_1$ and independent of $P$,
  and $P_{B_{1/u}^c}$ is the restriction of $P$ to the set
  $B_{1/u}^c$.  Since $r(u) = 1/u$ and
  $u\rho_k(t, A) = \rho_k(ut, uA)$,
  \begin{equation*} 
    \Tu[u] \Psi(P^{(u)}) = \Psi ( uP^{(u)}) \, ,
  \end{equation*}
  where 
  \begin{align*}
    uP^{(u)} = \tsum_{i=0}^k\delta_{U_i} + Z^{(u)} ,
  \end{align*}
  and $Z^{(u)}$ is $uP$ restricted to $B_1^c$.  Thus,
  for every nonnegative $f$,
  \begin{displaymath}
    L_f \big( \Tu \tX \,  \big| \, \xi>u  \big) \sim 
    L_f \big( \Psi ( \sY + Z^{(u)})  \big)\quad \text{as}\;  u\toi. 
  \end{displaymath}
  
  Take now $f:\E\to \R_+$ 
  whose support is contained in the set
  $B_a\times (0,\infty)\in \borel_{10}$ for some $a>0$. We can and
  will assume that $a\geq 2$. Observe that
  $$
  \pr(Z^{(u)} (B_a)\geq 1) \leq \ex [Z^{(u)} (B_a)]  = 
  \ex  [P (B_{a/u} \setminus B_{1/u})]  \leq 
  \frac{1}{u^d} C_d  a^d  \to 0 \, ,
  $$
  so that $\pr(Z^{(u)} (B_a) = 0)\to 1$ as $u \toi$. Since
  $f(t,s)=0$ for $t\notin B_a$, on the event
  $\{Z^{(u)} (B_a) = 0\}$ one has that
  \begin{align*}
    \int f \dint \Psi ( \sY + Z^{(u)})
    = \tsum_{i=0}^k f\big(U_i, \rho_k^{-1}(U_i,\sY + Z^{(u)})\big)
    = \tsum_{i=0}^k f\big(U_i, \rho_k^{-1}(U_i,\sY)\big) = \int f \dint \Psi ( \sY ),
  \end{align*}
  where the second equality follows since $a\geq 2$, and so the $k$th
  nearest neighbor of each $U_i$ is necessarily in $\sY$. Thus, for
  any such $f$,
  $$
  L_f \big( \Tu \tX \,  \big| \, \xi>u  \big) \sim 
  L_f \big( \Psi ( \sY + Z^{(u)})  \big)
  \to  L_f \big( \Psi ( \sY)  \big) = L_f(Y) \quad \text{as}\; u\toi.
  $$
  In particular, this holds for any nonnegative,
  continuous and bounded function $f$ on $\E$ whose support is in
  $\borel_{11}$ since $\borel_{11} \subset \borel_{10}$, and thus
  \eqref{eq:tail_process} holds.
\end{proof}

By \Cref{prop:S0t_is_Pareto}, the tail score at the origin $\eta$ is
$\mathrm{Pareto}(dk)$-distributed which is easily checked since
\begin{displaymath}
  \eta^{-1}=\rho_k(0,\sY)=\tmax_{1\leq i \leq k} \|U_i\|=:U^*. 
\end{displaymath}
By \Cref{prop:spectral_is_indep}, $\eta$ is independent of
the spectral tail configuration, which is (since $\beta=-1$) given by
\begin{align*}
  \Theta  :=T_{\eta^{-1}, \eta} Y
  &= \Big\{\big(U_i / U^*, \, \rho_k(U_i/U^*,\sY /U^*)^{-1}\big),i=0,\dots,k\Big\} \\
  &=\rev{\Psi(\{U_0/U^{*},\dots, U_k/U^{*}\})} \, .
\end{align*}
\rev{A direct calculation shows that the random set $\{U_0/U^{*},\dots, U_k/U^{*}\}$ has the same distribution as $\sY^*:=\{U_0,U_1,\dots, U_{k-1}, U'_k\}$, where $U'_k$ is uniformly
distributed on $\partial B_1$ and independent of $U_0, U_1,\dots, U_{k-1}$. In particular, }
\begin{equation}\label{eq:Theta_k_dist}
  \Theta \eind \Psi(\sY^*) =  \tsum_{i=0}^{k-1}
  \delta_{(U_i, \, \rho_k(U_i,\sY^*)^{-1})}
  + \delta_{(U_k', \,  \rho_k(U_k',\sY^*)^{-1})} \, .
\end{equation}


\subsubsection*{Point process convergence} 
\label{sub:small_dist_pp_conv}

First, (\ref{eq:small_dist_kNN_marginal}) implies that the sequence of
thresholds $(a_{\n})$ in (\ref{eq:a_n}) can be chosen as
\begin{align}\label{eq:a_n_small_kNN}
  a_{\n}:= \n^{1/k} \cdot \sqrt{\pi}(k!)^{-1/dk} \Gamma(d/2+1)^{-\frac{1}{d}}, \quad \n>0 \, .
\end{align}
Scaling scores with $a_{\n}^{-1}$ is (up to a transformation) equivalent to
scaling distances $\rho_k(t,P)$ with $a_{\n}$.

Let $A^{\mathrm{fm}}$ be defined at
\eqref{eq:first_maximum_anchor}. Then the extremal index (depending on
$k$ and dimension) is given by
\begin{align*}
  \vartheta_{k,d}:= \pr(A^{\mathrm{fm}}(\Theta)=0) \, , 
\end{align*}
and (\ref{eq:normalized_typical_cluster_cond}) implies that $Q$ has
the conditional distribution of $\Theta$ given that
$A^{\mathrm{fm}}(\Theta)=0$. 
Observe that $A^{\mathrm{fm}}(\Theta)=0$
if and only if $\sY^*$ is not contained in $\ball_1(U_i)$ for all
$i=1,\dots, k-1$, and $U_k'$ is lexicographically larger than $0$ if
$\sY^*$ is a subset of the closure of $\ball_1(U_k')$.

If $k=1$, then $\sY^*=\{0,U_1'\}$, so that $A^{\mathrm{fm}}(\Theta)=0$
if and only if $U_1'$ is lexicographically larger than
$0$. Thus, $\vartheta_{1,d}=1/2$  and
\begin{align*}
  Q \eind \Psi(\{0,U_1''\}) =  \delta_{(0, 1)} + \delta_{(U_1'', 1)}
\end{align*}
in all dimensions, where $U_1''$ is uniform on
$\partial B_1 \cap \{x=(x_1,\dots, x_d)\in \R^d : x_1\geq 0\}$. The
value $\vartheta_{1,d}=1/2$ is intuitively obvious since,
asymptotically, large values always come in pairs (with exactly the
same score). In dimension $d=2$, $\vartheta_{1,2}=1/2$ was obtained in
\cite[Section~4.2]{chen-rob18} when analyzing the extremal properties
of the inradius of a Poisson--Voronoi tessellation.

If $k=2$, $A^{\mathrm{fm}}(\Theta)=0$ if and only if
$U_1\notin \ball_1(U_2')$. Thus, in all dimensions,
\begin{align*}
  Q\eind \Psi(\{0,U_1',U_2'\}),
\end{align*}
where $U_1'$ is, conditionally on $U_2'$, uniform on
$\ball_1\setminus \ball_1(U_2')$. Furthermore, due to rotational invariance,
\begin{displaymath}
  \vartheta_{2,d}=1-\frac{\Leb(\ball_1\cap \ball_1(e_1))}{\Leb(\ball_1)}
  =1-\frac{2\Gamma(1+d/2)}{\sqrt{\pi}\Gamma((d+1)/2)}\int_0^{\pi/3}
  \sin^d u \;\dx u,
\end{displaymath}
where $e_1$ is the first basis vector of $\R^d$ and $\Leb(\cdot)$ is
the Lebesgue measure. In particular, $\vartheta_{2,1}=1/2$,
\begin{align*}
  \vartheta_{2,2} = \frac13 + \frac{\sqrt{3}}{2 \pi} \approx 0.609,
\end{align*}
and $\vartheta_{2,3}=33/48$.  Since
$\Gamma(x+\alpha)\sim\Gamma(x)x^\alpha$ as $x\toi$,\footnote{One can obtain the asymptotics for the intregral $I_d=\int_0^{\pi/3}
  \sin^d u \;\dx u$ by showing that 
  \begin{align*}
      \lim_{d\toi}\frac{d+1}{\sin^d(\pi/3)} I_d = \tan(\pi/3) \, .
  \end{align*}}
\begin{displaymath}
  1-\vartheta_{2,d}\sim \sqrt{\frac{2d}{\pi}} \cdot 2\frac{(\sqrt{3}/2)^{d+1}}{d+1}  \sim \; \sqrt{\frac{6}{\pi}} \cdot \frac{(\sqrt{3}/2)^{d}}{\sqrt{d}}  \, ,  \quad
  \text{as}\; d\toi \, ,
\end{displaymath}
i.e., $1-\vartheta_{2,d}$ goes to
zero exponentially fast as the dimension grows. 

For $d=1$, we have already seen that
$\vartheta_{1,1}=\vartheta_{2,1}=1/2$. Interestingly, one can check
that $\vartheta_{k,1}=1/2$ for all $k\in \N$. Indeed, one can assume
that $U'_k=1$ without loss of generality, and then the maximal score
is attained at zero if the unit ball around any of
$j\in\{0,\dots,k-2\}$ points that fall in $(0,1)$ does not cover $k-j$
points uniformly distributed in $(-1,0)$. This probability can be
calculated explicitly, and then the result follows by noticing that
$j$ is binomially distributed.

The exact calculations of $\vartheta_{k,d}$ become quite involved for 
$k\geq3$ and $d\geq 2$. 

\begin{proposition}
  \label{prop:small-dist-conv}
  For arbitrary $k, d\in \N$, define $X$ as in the beginning of
  \Cref{sub:small_dist_kth_NN} and $(a_{\n})_{\n}$ as in
  (\ref{eq:a_n_small_kNN}). For any 
  $(b_{\n})_{\n}$ such that $\n^{-1/k}/b_{\n}\to 0$ and
  $b_{\n}/\n \to 0$ as $\n \toi$, the assumptions of
  \Cref{thm:main_result} hold (with $\alpha=dk$ and $r(u)=u^{-1}$). Therefore, the
  convergence of the extremal blocks in (\ref{eq:pp_convergence})
  holds as well.
\end{proposition}
\begin{proof}
Assumption~\ref{hypo:t_n} holds, since $r(u)=u^{-1}$ and $a_\n$ has
the order $\n^{1/k}$. For each $\n>0$ and $\bi \in I_{\n}$ define the
block of indices $J_{\n,\bi}$ by (\ref{eq:blocks_of_indices}) and
$\bB_{\n,\bi}$ by \eqref{eq:extr-block}. The key step in checking
other assumptions of \Cref{thm:main_result} is that for every
$t\in P$, the condition $X(t)>y$ (i.e., $\rho_k(t,P)<y^{-1}$) is
equivalent to $P(\ball_{y^{-1}}(t))\geq k+1$.  Thus, given that
$X(t)>a_{\n} \epsilon$, $X(t)$ depends only on the points of $P$ in
$B_{a_{\n}^{-1} \epsilon^{-1}}(t)$, where $a_{\n}^{-1} \to 0$ as
$\n\toi$.

Fix $\eps,\delta,c>0$ and recall the notation from (\ref{eq:AC}). The
event $\{\xi>a_\n \eps\}= \{\rho_k(0,\tP)<(a_{\n} \eps)^{-1}\}$
depends only on $\tP$ (that is, $P$) restricted to
$\ball_{(a_{\n}\eps)^{-1}}$. Since
$C_{\n,u}\subset B_{(a_{\n}\eps)^{-1}u}^c$, as soon as
$u>1+ \delta^{-1}/\epsilon^{-1}$ the event
$\{M(\tX_{C_{\n,u}})>a_{\n} \delta\}=\{\min_{t\in \tP\cap C_{\n,u}}
\rho_k(t,\tP)<(a_{\n} \delta)^{-1}\}$ is determined only by $\tP$
restricted to $\ball_{(a_{\n}\eps)^{-1}}^c$ (on this set one has that
$P=\tP$ and, consequently, $X=\tX$). Since $P$ is a Poisson process,
this implies that for all such $u$,
\begin{align*}
  \pr\big(M(\tX_{C_{\n,u}})>a_\n\delta &\mid \xi>a_\n\eps\big)
  = \pr\big(M(\tX_{C_{\n,u}})>a_\n\delta)= \pr\big(M(X_{C_{\n,u}})>a_\n\delta)\\
  & \le \pr(M(X_{B_{b_{\n} c}})>a_{\n} \delta) \le \ex\left[\tsum_{t\in P'} \ind{\psi(t,P)>a_{\n}\delta, \, t\in B_{b_{\n}c}}\right] \\
  & = \ex\left[\tsum_{t\in P'} \ind{\psi(0,\shift{t}P)>a_{\n}\delta, \, t\in B_{b_{\n}c}}\right] \\
  & = \mathrm{const}\, (c b_{\n})^d \pr(\psi(0,\tP)> a_{\n}\delta) \\
  & =  \mathrm{const}\, \frac{b_{\n}^d}{\n^d} \n^d \pr(\xi> a_{\n}\delta) \to \mathrm{const} \cdot 0 \cdot \delta^{-\alpha}
  = 0  \quad \text{ as } \n\toi \, .
\end{align*}
In the sixth step we used the refined Campbell's theorem
(\ref{eq:refined_Campbell}), and in the penultimate step we used
(\ref{eq:a_n_eps}) and the fact that $b_{\n}$ is chosen such that
$b_{\n}/\n \to 0$. Thus, (\ref{eq:AC}) holds, i.e., \Cref{hypo:AC} is
satisfied.

Now consider Assumption~\ref{hypo:AI_of_extr_blocks}.  Take
$(l_{\n})_{\n}$ such that $l_{\n}/b_{\n}\to 0$ and
$a_{\n}^{-1}/l_{\n}\to 0$ as $\n \toi$; since
$a_{\n}^{-1}/b_{\n}=r(a_{\n})/b_{\n}\to 0$, one can, e.g., take
$l_{\n}=\sqrt{b_{\n}/a_{\n}^{-1}}$. Let $f_{\n, \bi}$, $\n>0$,
$\bi \in I_{\n}$, be an arbitrary family of shift-invariant measurable
functions from $\sN_{01}$ to $[0,\infty)$ such that for some
$\delta>0$ and for all $\n>0$, $\bi \in I_{\n}$ and $\mu\in \sN_{01}$,
\begin{displaymath}
  f_{\n, \bi}(\mu)=
  \begin{cases}
    0 & \text{if}\; M(\mu)\leq \delta,\\
    f_{\n, \bi}(\mu^{\delta}) & \text{otherwise} \, ,
  \end{cases}
\end{displaymath}
where $\mu^{\delta}$ denotes the restriction of $\mu$ to
$\R^d\times (\delta,\infty)$.

Similarly as above, for $t\in P$, the random variable
$X(t)\ind{X(t)>a_{\n} \delta}$ depends only on $P$ restricted to
$B_{(a_{\n}\delta)^{-1}}(t)$. Moreover, if
$(a_{\n}\delta)^{-1}<l_{\n}$, then
\begin{align*}
  \bigcup_{t\in \hat{J}_{\n,\bi}} B_{(a_{\n}\delta)^{-1}}(t) \subseteq J_{\n,\bi}
\end{align*}
for all $\bi\in I_{\n}$; recall that $\hat{J}_{\n,\bi}$ in
(\ref{eq:trimmed_blocks}) are obtained from the original blocks
$J_{\n,\bi}$ by trimming the edges by $l_{\n}$. Thus, since
$\widehat{\bB}_{\n,\bi}:= X_{\hat{J}_{\n,\bi}}$, the value of
$f_{\n,\bi}(\Tu[a_{\n}]\widehat{\bB}_{\n,\bi})$ depends only on $P$
restricted to $J_{\n,\bi}$, for all $\bi\in I_{\n}$. For such $\n$,
since $P$ is a Poisson process and $J_{\n,\bi}$'s are disjoint, the
left-hand side of (\ref{eq:AI_blocks_cond}) vanishes, hence,
\Cref{hypo:AI_of_extr_blocks} holds.
\end{proof}

A simple consequence of Proposition~\ref{prop:small-dist-conv}
concerns the behavior of the minimal distance to the $k$-th nearest
neighbour in a Poisson configuration of points on an increasing
hypercube $[0,\n]^d$.
For all $\n>0$, denote $m_{\rho,\n} = \min \{\rho_k(t,P) : t \in P\cap [0,\n]^d\}$. Then, for
$a_\n$ as in \eqref{eq:a_n_small_kNN}, (\ref{eq:easy_consq}) implies
$$
 \pr ( a_\n m_{\rho,\n} \ge v) \to e^{-\vartheta_{k,d}
   v^{dk}}\quad
 \text{as}\; n\to\infty
$$
for any $v>0$, that is, the scaled minimum distance to the $k$th
nearest neighbour in a Poisson configuration in $[0,\n]^d$ converges
to a (scaled) Weibull distribution.

\subsection{Moving maxima model} 
\label{sub:moving_maxima_tail}

\rev{Moving maxima models are very popular in the studies of time
  series. Below we put such models in the spatial context by starting
  with a marked Poisson point process and then adjusting the scores by
  giving each point the score derived from weighted maximum of the
  scores of points from its neighbours. This general construction
  was described in Section~\ref{sec:gener-constr-scor}.} 

Let $\zeta$ be a positive random variable whose tail is regularly
varying of index $\alpha>0$, and denote the distribution of $\zeta$ by
$m$. Assume that $P\in \sN_{10}$ is a unit rate independently marked
stationary Poisson process on $\R^d$ with marks in $(0,\infty)$, whose
intensity measure is the product of the Lebesgue measure on $\R^d$ and
a probability measure $m$, so that $\zeta$ is the typical mark. In the
sequel, write $P=\tsum \delta_{(t,\zeta_t)}$ where $(\zeta_t)_{t}$,
$t\in\R^d$, are i.i.d.\ with distribution $m$. \rev{While, by
  independence, the large values of $\zeta_t$ come in isolation in
  $P$, the clustering of scores is easily modeled by
  considering, for instance,
  $X=\tsum \delta_{(t,\max_{|t-s|<1} \zeta_s)}$, where the ``large''
  scores in $P$ propagate to all neighbouring locations in the ball of
  radius 1.}

Let $\sN_g$ be the space of all simple locally finite point measures
on $\R^d$ equipped with the usual $\sigma$-algebra. Consider a
measurable function $\Phi:\R^d\times \sN_g\to \sN_g$ such that, for
each $\mu'\in \sN_g$, and each $t\in \mu'$, $\Phi(t,\mu')$ is a
finite subset of $\mu'$ which contains $t$. It is useful to interpret
$\Phi(t,\mu')$ as a \emph{neighborhood} of $t$ in $\mu'$, and all
points $x\in \Phi(t,\mu')$ as \emph{neighbours} of $t$. Furthermore,
assume that $\Phi$ is shift-equivariant in the sense that
$\Phi(t-x, \shift{x}\mu')=\Phi(t,\mu')-x$ for all $\mu'\in \sN_g$,
$t\in \mu'$ and $x\in \R^d$. Finally, denote by $N(t,\mu')$ the
cardinality of $\Phi(t,\mu')$. It can been easily seen that the
arguments below also work if $\Phi$ depends on some external sources
of randomness, but stays independent of the marks $(\zeta_t)$.

Consider a (deterministic) weight function $w:\R^d\to [0,\infty)$ such
that $w(0)>0$, and define the scoring function
$\psi:\R^d \times \sN_{10}\to [0,\infty)$ by
\begin{align*}
  \psi(t, \mu):= {\tmax_{x\in \Phi(t,\mu')}} w(x-t) \mu(x),
  \quad   t\in\mu',
\end{align*}
where $\mu'$ \rev{is the projection of $\mu$ on $\R^d$}.
Without loss of generality,
in the sequel assume that $w(0)=1$. If the weight function is
identically one, the score at $t$ is the maximum of the scores of its
neighbours.

Since $\psi$ is shift-invariant, $X:=\Psi(P)$ is a stationary marked
point process on $\R^d$. Its Palm version is given by
$\tX:=\Psi(\tP)$, where $\tP=P+\delta_{(0,\zeta_0)}$, with $\zeta_0$
having distribution $m$ and being independent of $P$. For notational
convenience, denote $\tPhi(t)=\Phi(t,\tP')$, and let
$\tN(t)=N(t,\tP')$ be the cardinality of $\tPhi(t)$ for all
$t\in \tP'$. \rev{Clearly,
  $X=\tsum \delta_{(t,\max_{|t-s|<1} \zeta_s)}$ is a special example
  of this construction with $w\equiv 1$ and
  $\Phi(t,\mu') = \tsum_{s \in \mu'} \delta_s \one{|s-t|<1}$, i.e.,
  for the neighbourhoods which consist of all points $s \in \mu'$
  within distance less than $1$ to $t$.}

\subsubsection*{Tail configuration}

\rev{Recall that the random score $\zeta$ is assumed to be regularly
  varying with tail index $\alpha$.}

\begin{theorem}
  \label{thm:tail_conf_moving_maxima}
  Assume that $w$ is a bounded function and that
  $\ex [\tN(0)]<\infty$. Then the following statements hold.
  \begin{itemize}
  \item[(i)] The score at the origin $\xi$ satisfies
    \begin{align}\label{eq:mov_max_xi_is_RV}
      \lim_{u\toi}\frac{\pr(\xi>u)}{\pr(\zeta>u)}
      \to  \ex\left[\tsum_{t\in \tPhi(0)} w(t)^\alpha \right]
      =:\kappa\in (0,\infty) \, .
    \end{align}
    In particular, $\xi$ is regularly varying with tail index
    $\alpha$.
  \item[(ii)] The tail configuration of $X$ exists for a constant
    scaling function $r\equiv 1$ (thus, with no scaling of the
    positions) and the distribution of its spectral part
    $\Theta\in \sN$ is given by
    \begin{align}\label{eq:Theta_moving_maxima}
      \ex[h(\Theta)]=\kappa^{-1} \ex\Big[\tsum_{x\in \tPhi(0)}
      h\Big(\big\{\big(t,\frac{w(x-t)}{w(x)}\big)
      : t\in \tP',\; x\in \tPhi(t)\big\}\Big) w(x)^{\alpha}\Big] 
    \end{align}
    for all measurable $h:\sN\to \R_+$, where the summand
    corresponding to $x$ is understood to be $0$ if $w(x)=0$.
  \end{itemize}
\end{theorem}

\begin{remark}\label{rem:moving_maxima_notation} 
  For a point measure $\mu$ on $\R^d\times [0,\infty)$ whose
  restriction to $\E=\R^d\times (0,\infty)$ is an element of
  $\sN$, for any function $h$ on $\sN$ we define $h(\mu)$ to be equal
  to the value of $h$ on this restriction. In other words, we simply
  neglect the points of the form $(t,0)$. This is relevant for
  (\ref{eq:Theta_moving_maxima}) since the weighting function $w$ can
  in general attain the value 0.
\end{remark}

\begin{remark}\label{rem:tail_pr_intepretation}
  Observe that the distribution of $\Theta$, in addition to $w, \alpha$ and
  $\Phi$, depends only on the distribution of $\tP'= P' +
  \delta_0$. It can be obtained as follows. First, let $P^*$ be
  distributed as $\tP'$ but from the tilted distribution:
  \begin{align}\label{eq:P_star}
    \pr(P^* \in  \cdot \,)
    := \kappa^{-1}\, \ex\left[\one{\tP' \in \, \cdot}
    \, \tsum_{x\in \tPhi(0)} w(x)^{\alpha} \right] \, ,
  \end{align}
  and denote $\Phi^*(t):=\Phi(t,P^*)$ for all $t\in
  \R^d$. Conditionally on $P^{*}$, let $V$ be a $\Phi^*(0)$-valued
  random element such that $V$ equals $x\in \Phi^*(0)$ with
  probability proportional to $w(x)^\alpha$. Finally, let
  \begin{displaymath}
    \Theta:=\Big\{\big(t, \tfrac{{w}(V-t)}{{w}(V)}\big) : t\in P^{*},\; V\in
    \Phi^*(t) \Big\}
  \end{displaymath}
  restricted to $\E$.
\end{remark}

\begin{example}\label{exa:special_case}
Here we use the notation of \Cref{rem:tail_pr_intepretation}.
\begin{itemize}
    \item[(i)] Assume that $w(t)=0$ for all $t\neq 0$ (and $w(0)=1$), so that
   $\psi(t, P) = \zeta_{t}$.  Then  $P^{*} \eind \tP'$ and $V=0$ almost surely, so $\Theta=\{(0,1)\}$,
  i.e., the extreme scores of $X$ appear in isolation.
  \item[(ii)] {Let now $w(t)\equiv1$, so that
  \begin{align*}
      \psi(t, P) =\tmax_{x\in \Phi(t,P')}\zeta_x \, , 
  \end{align*}
    and assume that $\ex \tN(0)<\infty$. Then (\ref{eq:P_star}) implies that $P^*$ is, compared to $\tP'$, biased towards configurations in which the origin has more neighbors. Furthermore,
  $\Theta = \{(t, 1) : t\in P^*, V\in \Phi^*(t) \}$, where $V$ is
  uniform on $\Phi^*(0)$. Note that necessarily $(0,1)\in
  \Theta$.} Observe also that (\ref{eq:mov_max_xi_is_RV}) implies that
  \begin{align}\label{eq:mov_max_xi_RV_special_case}
    \lim_{u\toi}\frac{\pr(\tmax_{x\in \tPhi(0)} \zeta_x >u)}
    {\pr(\zeta>u)} = \ex[\tN(0)] \, . 
  \end{align}
\end{itemize}
\end{example}

\begin{proof}[Proof of \Cref{thm:tail_conf_moving_maxima}]
  Let $f:\E\to [0,\infty)$ be an arbitrary continuous function such
  that
  \begin{align}\label{eq:moving_maxima_aux0}
    f(t,s)=0 \text{  whenever  } t\notin B_a
    \text{ or } s\leq \epsilon \, , 
  \end{align} 
  for some $a,\epsilon>0$. We extend $f$ to a continuous function
  on $\R^d\times [0,\infty)$ by letting $f(t,0)=0$ for all $t$. For
  notational convenience set $h(\mu):= e^{-\mu(f)}$ for all
  $\mu \in \sN_{11}$.  \rev{Both (i) and (ii)} would follow immediately if we show that
  \begin{align}\label{eq:moving_maxima_aux1}
    \lim_{u\toi} \frac{\ex[h(T_{1,u} \tX)\one{\tX(0)>u}]}
    {\pr(\zeta>u)} = \kappa \ex[h(Y)] \, ,
  \end{align}
  where $Y:= T_{1,\eta^{-1}}\Theta$ for $\Theta$ from
  (\ref{eq:Theta_moving_maxima}) and $\eta$ is $\mathrm{Pareto}(\alpha)$
  distributed and independent of $\Theta$.

  First, write $\tP=\{(t_{i}, \zeta_{i}) : i\geq 1\}$, where
  $\{t_1,\dots, t_{\tN(0)}\} = \tPhi(0)$.  Given \rev{the projection
  $\tP'$ of $\tP$}, the score
  $\psi(t,\tP)$ depends only on $\zeta_x$ for $x\in \tPhi(t)$. By
  (\ref{eq:moving_maxima_aux0}),
  \begin{align}
    h(T_{1,u}\tX)
    &= \exp\left\{-\tsum_{(t,s)\in \tX} f(t,s/u)\right\}
      = \exp\left\{- \tsum_{(t,s)\in \tX_{B_a}}
      f(t,s/u)\right\} \nonumber \\ 
    & = \exp\left\{- \tsum_{i=1}^{K_a}
      f(t_{i},\psi(t_i, \tP)/u)\right\}=:
      g\big(\tP', \zeta_1/u,\zeta_2/u,\dots, \zeta_{K_a}/u\big)  \, ,
      \label{eq:moving_maxima_aux2}
  \end{align}
  where $K_a=K_a(\tP')$ is the smallest nonnegative integer such that
  $\{t_1,\dots, t_{K_a}\}$ contains $\tPhi(t)$ for all
  $t\in \tP'_{B_a}$.
  Observe that $K_a\geq 1$.
  Since $\tP'$ and $(\zeta_1,\zeta_2,\ldots)$ are independent,
  conditioning on $\tP'$ yields that
  \begin{align}\label{eq:moving_maxima_aux2dot5}
    \ex\Big[h(T_{1,u}\tX)\one{\tX(0)>u}\Big] = \ex\big[g_u(\tP')\big] \, ,
  \end{align}
  where, for all $\mu'$ from the support of $\tP'$, 
  \begin{align}\label{eq:moving_maxima_aux2dot75}
    g_u(\mu') := \ex\Big[g\big(\mu', \zeta_1/u,\dots,
    \zeta_{K_a(\mu')}/u\big) \ind{\tmax_{i\leq N(0,\mu')}w(t_i) \zeta_i>u}\Big], 
  \end{align}
  and $(t_i)_{i\geq1}$ are deterministic and depend only on $\mu'$.  Observe
  that for every fixed $\mu$ (write $k:=K_a(\mu')$ and $n:=N(0,\mu')$,
  so $k\geq n$), the function under the expectation
  \begin{align*}
    (y_1,\dots,y_{k}) \mapsto g(\mu, y_1,\dots, y_{k})
    \ind{\tmax_{i\leq n}w(t_i) y_i>1}
  \end{align*}
  is bounded (since $h$ is bounded) and continuous except on the set
  \begin{displaymath}
    \big\{(y_1,\dots, y_{k}): \tmax_{i\leq n}w(t_i) y_i=1\big\}
  \end{displaymath}
  (since $f$ and hence $h$ is continuous). Furthermore, this function
  has support bounded away from the origin in $\R^{k}$, since it
  vanishes whenever
  \begin{align*}
    \max\{y_1, y_2,\dots, y_n\} \leq 1 / \tmax_{i\leq n}w(t_i) \, .
  \end{align*}

  Since $(\zeta_{i})_{i}$ are i.i.d.\ regularly varying with index
  $\alpha$, the vector $(\zeta_{1},\dots, \zeta_{k})$ is multivariate
  regularly varying in $\R^k_+\setminus\{0\}$ with the same index, see
  \cite[p.~192]{resnick:2007}. In particular,
  \begin{multline}\label{eq:moving_maxima_aux3}
    \lim_{u\toi}\frac{g_u(\mu')}{\pr(\zeta>u)} 
    = \int_{\R_+^{k}\setminus\{0\}}g(\mu', y_1,\dots, y_{k})
    \ind{\tmax_{i\leq n}w(t_i) y_i>1} \, \dx\nu(y_1,\dots,y_{k} )
  \end{multline}
  for a certain measure $\nu$ on $\R_+^{k}\setminus\{0\}$ concentrated
  on the axes. More precisely, by \cite[p.~192]{resnick:2007}, if
  $\eta$ is $\mathrm{Pareto}(\alpha)$ distributed, the right-hand side
  of (\ref{eq:moving_maxima_aux3}) equals
  \begin{multline*}
    \tsum_{i=1}^n\int_{0}^{\infty}g(\mu',
    \underbrace{0,\dots,0}_{i-1},y,0,\dots,0)\one{w(t_i)y>1}    
    \alpha y ^{-\alpha-1} \dx y  \\
    = \tsum_{i=1}^n \ex\big[g(\mu',
    \underbrace{0,\dots,0}_{i-1},\eta/w(t_i),0,\dots,0)\big] w(t_i)^\alpha,
  \end{multline*}
  where the $i$th summand on the right-hand side is set to be 0 if
  $w(t_i)=0$.  Recalling that $g$ was defined at
  (\ref{eq:moving_maxima_aux2}),
  this equals
  \begin{align*}
    \ex \Big[\tsum_{i=1}^n \exp\left\{-\tsum_{j\leq k \, :
    \, t_i\in \Phi(t_j,\mu')} f\big(t_j,
    \frac{w(t_i-t_j)}{w(t_i)}\eta \big) -\tsum_{j\leq k \, :
    \, t_i\notin \Phi(t_j,\mu')} f(t_j,0) \right\} w(t_i)^\alpha \Big]\, . 
  \end{align*}
  Recall that $\{t_1,\dots, t_n\}=\Phi(0,\mu')$ and that $n\leq
  k$. Since $f(t,0)=0$ for all $t\in \R^d$ and since for
  $j>k=K_a(\mu')$, one has $t_j\notin B_a$ and 
  $f(t_j, y)=0$ regardless of $y$, the expression above (and therefore
  the right-hand side of (\ref{eq:moving_maxima_aux3})) actually
  equals
  \begin{align*}
    \ex \Big[\tsum_{x\in \Phi(0,\mu')} \exp\left\{-\tsum_{t \in \mu' \, :
    \, x \in \Phi(t,\mu')} f\big(t, \frac{w(x-t)}{w(x)}\eta \big)
    \right\} 
    w(x)^\alpha \Big] \, .
  \end{align*} 
  Going back to (\ref{eq:moving_maxima_aux2dot5}),
  (\ref{eq:moving_maxima_aux3}) yields
  \begin{multline}\label{eq:moving_maxima_aux4}
    \lim_{u\toi} \frac{\ex[h(T_{1,u} \tX)
      \one{\tX(0)>u}]}{\pr(\zeta>u)}
    = \lim_{u\toi}\ex\Big[\frac{g_u(\tP')}{\pr(\zeta>u)} \Big]
    = \ex\Big[\lim_{u\toi}\frac{g_u(\tP')}{\pr(\zeta>u)} \Big] \\
    = \ex\big[ \tsum_{x\in \Phi(0,\tP')} \exp\left\{-\tsum_{t \in \tP' \, :
        \, x \in \tPhi(t)} f\big(t, \frac{w(x-t)}{w(x)}\eta \big)
    \right\} w(x)^\alpha\Big]\, ,
  \end{multline}
  where $\eta$ and $\tP'$ are independent, which is precisely
  (\ref{eq:moving_maxima_aux1}). It remains to justify the interchange
  of the limit and expectation in (\ref{eq:moving_maxima_aux4}).

  Since $g$ is bounded by $1$ and $w_*:=\sup_{t\in \R^d} w(t)<\infty$, for
  each $\mu'$ and $u>0$, we have $g_u(\mu') \leq N(0,\mu')
  \pr(\zeta>u/w_*)$. The regular variation property of  $\zeta$ yields
  that
  \begin{align*}
    0\leq \frac{g_u(\tP')}{\pr(\zeta>u)} \leq
    \tN(0) \frac{\pr(\zeta>u/w_*)}{\pr(\zeta>u)} \to \tN(0) w_{*}^\alpha
  \end{align*}
  almost surely as $u\toi$. Moreover, since
  $\ex[\tN(0)w_*^\alpha]=\ex[\tN(0)]w_*^\alpha $ is finite by
  assumption, Pratt's extension of the dominated convergence theorem
  (see \cite[Theorem~1]{pratt:1960}) justifies the interchange
  \begin{align*}
    \lim_{u\toi}\ex\Big[\frac{g_u(\tP')}{\pr(\zeta>u)} \Big]
    = \ex\Big[\lim_{u\toi}\frac{g_u(\tP')}{\pr(\zeta>u)} \Big] \, ,
  \end{align*}
  and this finishes the proof.
\end{proof}

\subsubsection*{Point process convergence}
\label{sub_moving_maxima_Q_theta}

In the following assume that the assumptions of
\Cref{thm:tail_conf_moving_maxima} hold so that $X$ admits a tail
configuration $Y$ whose spectral configuration $\Theta$ is given by
(\ref{eq:Theta_moving_maxima}); recall the convention explained in
\Cref{rem:moving_maxima_notation}. 
In order to determine
the ingredients of the limiting point process in
\Cref{thm:main_result}, the following result is crucial; our approach used here is similar to the one taken by \cite[Section~4]{planinic:2021} for random fields over $\Z^d$.

\begin{lemma}
  Denote by $W$ the random element in $\sN$ equal to 
  \begin{align}\label{eq:mov_max_W}
    \big\{(t,w(-t)) : t\in \tP', 0\in \tPhi(t)\big\}
  \end{align}
  restricted to $\E$. Then the spectral tail configuration $\Theta$
  from (\ref{eq:Theta_moving_maxima}) satisfies
  \begin{align}\label{eq:mov_max_Theta_using_W}
    \ex[h(\Theta)] 
    = \kappa^{-1} \ex\left[\tsum_{(t,s)\in W}
    h(T_{1,s} \shift{t}W) s^\alpha \right]
  \end{align}
  for all $h:\sN\to [0,\infty)$. In particular, 
  \begin{align}\label{eq:mov_max_kappa_alt}
    \kappa=\ex\left[\tsum_{(t,s)\in W} s^\alpha\right].
  \end{align}
\end{lemma}
\begin{proof}
  Observe that
  \begin{align*}
    \kappa \ex[h(\Theta)]
    &= \ex\Big[\tsum_{x\in \tP'}
      h\Big(\big\{\big(t,\frac{w(x-t)}{w(x)}\big)
      : t\in \tP',\; x\in \tPhi(t)\big\}\Big)\one{x\in \tPhi(0)} w(x)^{\alpha}\Big] \\
    & =  \ex\Big[\tsum_{x\in \tP'}
      h\Big(\big\{\big((t-x)+x,\frac{w(-(t-x))}{w(x)}\big)
      : t-x\in \shift{x}\tP',\; 0\in \Phi(t-x,\shift{x}\tP')\big\}\Big) \\
    & \qquad\qquad \qquad\qquad \qquad\qquad \qquad\qquad \qquad\qquad
      \times \ind{0\in \Phi(-x,\shift{x}\tP')} w(x)^{\alpha}\Big] \\
    & =  \ex\Big[\tsum_{x\in \tP'}
      h\Big(\big\{\big(t+x,\frac{w(-t)}{w(x)}\big)
      : t\in \shift{x}\tP',\; 0\in \Phi(t,\shift{x}\tP')\big\}\Big)
      \one{0\in \Phi(-x,\shift{x}\tP')} w(x)^{\alpha}\Big] \\
    &=: \ex\Big[\tsum_{x\in \tP'}
      g(-x,\shift{x}\tP')\Big] \, ,
  \end{align*}
  where we used the shift-equivariance of $\Phi$ to obtain the second
  equality. The point-stationarity of $\tP'$ (see (\ref{eq:mecke}))
  yields that 
  \begin{align*}
    \kappa \ex[h(\Theta)]
    &= \ex\Big[\tsum_{x\in \tP'}g(x,\tP')\Big] \\
    & =   \ex\Big[\tsum_{x\in \tP'}
      h\Big(\big\{\big(t-x,\frac{w(-t)}{w(-x)}\big)
      : t\in \tP',\; 0\in \Phi(t,\tP')\big\}\Big)
      \one{0\in \Phi(x,\tP')} w(-x)^{\alpha}\Big] \\
    &= \ex\Big[\tsum_{x\in \tP'}
      h\Big(T_{1,w(-x)} \shift{x}W\Big)\one{0\in \tPhi(x)} w(-x)^{\alpha}\Big] \, ,
  \end{align*}
  which is precisely what we wanted to prove. Expression
  (\ref{eq:mov_max_kappa_alt}) follows by taking $h\equiv 1$.
\end{proof}

{By setting $w \equiv 1$, the two expressions for $\kappa$ in (\ref{eq:mov_max_xi_is_RV}) and (\ref{eq:mov_max_kappa_alt}) imply that
\begin{align*}
    \ex\left[ \tsum_{t\in \tP'} \one{0\in \tPhi(t)}  \right] = \ex [\tN(0)]<\infty \, .
\end{align*}
This implies that $W$, and therefore $\Theta$, almost surely has finitely many points in
$\E=\R^d\times (0,\infty)$. In particular, we can regard $W$ and $\Theta$ as elements of $\sN_{01}$, i.e.\  $\pr(W\in \sN_{01})=\pr(\Theta\in \sN_{01})=1$; the same holds for the tail configuration as well.}

We now turn our attention to the process $\bQ$ defined in (\ref{eq:normalized_typical_cluster}). Recall that $M(\mu)$ denotes the maximal score of $\mu\in \sN_{01}$.

\begin{proposition}
  The distribution of $\bQ$ from (\ref{eq:normalized_typical_cluster})
  in $\tsN_{01}$ is given by
  \begin{align}\label{eq:moving_maxima_Q}
    \pr([Q]\in \cdot\, ) = \frac{1}{\ex[M(W)^\alpha]}
    \ex\Big[\one{T_{1,M(W)}[W]\in \cdot} \; M(W)^\alpha\Big] \, .
  \end{align}
  Moreover, 
  \begin{align}\label{eq:moving_maxima_extremal}
    \vartheta= \frac{\ex\big[\tmax_{(t,s)\in W} s^{\alpha}\big]}
    {\ex\left[\tsum_{(t,s)\in W} s^{\alpha}\right]} \, .
  \end{align}
\end{proposition}
\begin{proof}
  Using (\ref{eq:normalized_typical_cluster_lo}) and since $\beta=0$,
  for an arbitrary shift-invariant and bounded
  $h:\sN_{01}\to [0,\infty)$, we have
  \begin{align*}
    \vartheta \ex [h(Q)]
    &= \ex \Big[h(T_{1,M(\Theta)}\Theta)\;
      \frac{M(\Theta)^{\alpha}}
      {\tsum_{(t,s)\in\Theta}s^\alpha} \Big]  = : \ex[\tilde{h}(\Theta)] \, .
  \end{align*}
  Observe that $\tilde{h}$ is shift-invariant and homogenous
  in the sense that $\tilde{h}(T_{1,y}\mu)=\tilde{h}(\mu)$ for all
  $\mu\in\sN_{01}$, $y>0$. By (\ref{eq:mov_max_Theta_using_W}),
  \begin{align*}
    \vartheta \,\ex [h(Q)]
    & = \kappa^{-1} \ex\left[\tsum_{(t,s)\in W} \tilde{h}(T_{1,s} \shift{t}W) s^\alpha\right] \\
    &= \kappa^{-1} \ex\left[\tsum_{(t,s)\in W} \tilde{h}(W) s^\alpha\right] \\
    &=  \kappa^{-1} \ex\left[\tsum_{(t,s)\in W} {h}(T_{1,M(W)} W)
      \tfrac{M(W)^\alpha}{\tsum_{(x,y)\in W} y^\alpha } s^\alpha\right] \\
    &= \kappa^{-1} \ex\left[{h}(T_{1,M(W)} W) M(W)^\alpha
      \tfrac{\tsum_{(t,s)\in W}s^\alpha}{\tsum_{(x,y)\in W} y^\alpha } \right] \\
    &=\kappa^{-1} \ex\left[{h}(T_{1,M(W)} W) M(W)^\alpha \right] \, .
  \end{align*}
  Due to (\ref{eq:mov_max_kappa_alt}), taking $h\equiv 1$ yields
  (\ref{eq:moving_maxima_extremal}), while (\ref{eq:moving_maxima_Q})
  follows since $h$ was arbitrary.
\end{proof}

We now give sufficient conditions under which the assumptions of
\Cref{thm:main_result} are satisfied. First, take a family
$(a_{\n})_{\n}$ such that (\ref{eq:a_n}) holds, and fix an arbitrary
family $(b_{\n})_{\n}$ such that $b_{\n} \toi$ and $b_{\n}/\n \to 0$
as $\n \toi$; since $r$ is a constant function, this is equivalent to
choosing $(b_{\n})$ such that \Cref{hypo:t_n} holds.



\begin{hypothesis}\label{hypo:stabilization}
Let $\tsN_g$ be the space of all simple locally finite point measures $\mu'$
on $\R^d$ such that $0\in \mu'$. Assume that there exists a
measurable function $R:\tsN_g\to [0,\infty]$ such that $R(\tP')<\infty$ a.s.\ and for all $\mu'\in \tsN_g$,
\begin{itemize}
    \item[(i)] for all $t\in \mu'$ such that $t\notin B_{R(\mu')}$,  $\Phi(0,\mu')\cap \Phi(t,\mu')=\emptyset$ and $\Phi(t,\mu') = \Phi(t,\mu'\setminus\{0\})$;
  \item[(ii)] $\Phi(0,\mu') = \Phi(0,\nu')$ for all $\nu'\in \sN_g$ such that $\mu'$ and $\nu'$ coincide on $B_{R(\mu')}$, i.e.\ $\Phi(0,\mu')$ is unaffected by changing points in $\mu'$ outside of $B_{R(\mu')}$.
\end{itemize}
\end{hypothesis}

Note that (ii) above necessarily implies that $\Phi(0,\mu') = \Phi(0,\mu'\cap B_{R(\mu')}) \subseteq B_{R(\mu')}$,  for all $\mu'\in \tsN_g$.


\begin{example}
    \begin{itemize}
        \item[(a)]  Assume that for some $r_0>0$,
        \begin{align*}
           \rev{ \Phi(t,\mu') = \{x\in \mu' : |t-x|<r_0\} \, , \; \;  \text{for all $t\in \mu'$} \, .} 
        \end{align*}
        In this case \Cref{hypo:stabilization} is clearly satisfied if we take $R= 2r_0$.
        
        \item[(b)]  If $\Phi(t,\mu')$ is the set containing $t$ and the $k$
nearest neighbors of $t$ in $\mu'\setminus \{t\}$ (with respect to the
Euclidean distance), $R$ satisfying \Cref{hypo:stabilization} can be
constructed as in the proof of \cite[Lemma~6.1]{penrose:yukich:2001},
see also \cite[p.~104]{eichelsbacher:2015}. Observe that in this case taking $R(\mu')$ to be the distance to the $k$-th nearest neighbor of $0$ in $\mu'\setminus \{0\}$ is not sufficient for property (i), and this property is crucial to ensure that the anticlustering condition holds.

  \end{itemize}
\end{example}

\begin{proposition}\label{prop:mov_max_ac}
  If \Cref{hypo:stabilization} holds, the anticlustering condition
  (\ref{eq:AC}) also holds.
\end{proposition}
\begin{proof}
{Denote $\tilde{R}:=R(\tP')$} and $w_*:=\max_{x\in \R^d} w(x)$. Recall that
  $w_*\in [1,\infty)$ since $w(0)=1$ and $w$ is a bounded
  function. For notational convenience, assume that $w_{*}=1$; the
  proof below is easily extended to the general case. Then, for every
  $t\in \tP'$,
  \begin{align*}
    \psi(t,\tP)\leq  \tmax_{x\in \Phi(t,\tP')}  \zeta_x \, .
  \end{align*}
  Due to (\ref{eq:mov_max_xi_is_RV}), for arbitrary
  $\epsilon,\delta,c>0$, to show (\ref{eq:AC}) it suffices to prove
  that
  \begin{align}\label{eq:ac_moving_maxima_aux1}
    \lim_{u\toi}\limsup_{\n \toi} \frac{\pr(A_{\n,u})}{\pr(\zeta>a_{\n}\epsilon)} = 0 \, ,
  \end{align} 
  where 
  \begin{align*}
    A_{\n,u}:=\left\{\tmax_{t\in \tP'\cap C_{\n,u}, x\in
    \tPhi(t)} \zeta_x> a_{\n}\delta,
    \; \tmax_{x\in \tPhi(0)} \zeta_x >a_{\n}\epsilon\right\} \, .
  \end{align*}
  Fix a $u>0$.
  If $\tilde{R}<u$, since $C_{\n,u} \subseteq B_u^c$, \Cref{hypo:stabilization}(i) implies that the families
  $\{ \zeta_x : x\in \tPhi(t) \text{ for some } t\in \tP'\cap
  C_{\n,u}\}$ and $\{\zeta_x : x\in \tPhi(0)\}$ consist of completely
  different sets of the $\zeta_x$'s, hence are independent (given
  $\tP'$). {Moreover, if $\tilde{R}<u$, since $\tP'=P'\cup \{0\}$, \Cref{hypo:stabilization}(i) also implies that}
  \begin{align*}
    \tmax_{t\in \tP'\cap C_{\n,u}, x\in \tPhi(t)} \zeta_x
    = \tmax_{t\in P'\cap C_{\n,u}, x\in \Phi(t, P')} \zeta_x \, .
  \end{align*}
  
  Therefore, 
  \begin{align*}
    \pr(A_{\n,u},\tilde{R}<u \mid \tP')
    &= \one{\tilde{R}<u}  \pr(\tmax_{t\in P'\cap C_{\n,u}, x\in
      \Phi(t,P')} \zeta_x> a_{\n}\delta \mid \tP')
      \pr(\tmax_{x\in \tPhi(0)} \zeta_x >a_{\n}\epsilon \mid \tP') \\
    &\leq  \pr(\tmax_{t\in P'\cap C_{\n,u}, x\in \Phi(t,P')}
      \zeta_x> a_{\n}\delta \mid \tP') \tN(0) \pr(\zeta>a_{\n}\epsilon) \, ,
  \end{align*} 
  where the inequality relies on the fact that $\tP'$ is independently
  marked. Furthermore,
  \begin{align*}
    \pr(A_{\n,u},\tilde{R}\geq u \mid \tP')
    \le \one{\tilde{R}\geq u}  \pr(\tmax_{x\in \tPhi(0)} \zeta_x
    >a_{\n}\epsilon \mid \tP') \le \one{\tilde{R}\geq u}\tN(0) \pr(\zeta>a_{\n}\epsilon) \, .
  \end{align*}
  Conditioning on $\tP'$ yields that
  \begin{align*}
    \frac{\pr(A_{\n,u})}{\pr(\zeta>a_{\n}\epsilon)}
    \leq \ex\left[\ind{\tmax_{t\in P'\cap C_{\n,u}, x\in
    \Phi(t, P')} 
    \zeta_x> a_{\n}\delta} \tN(0) \right]  + \ex\left[\one{\tilde{R}\geq u} \tN(0) \right]
  \end{align*}
  for all $u,\n>0$. The second term on the right-hand side does not
  depend on $\n$ and vanishes as $u\toi$ by the dominated convergence
  theorem since $\pr(\tilde{R}<\infty)=1$ and $\ex \tN(0)<\infty$. For the
  first term, observe that for each $u>0$ and $n\in \N$,
  \begin{multline*}
    \limsup_{\n\toi} \ex\left[\ind{ \tmax_{t\in P'\cap
          C_{\n,u}, x\in \Phi(t, P')}
        \zeta_x> a_{\n}\delta} \tN(0) \right] \\
    \le n \limsup_{\n\toi}\pr(\tmax_{t\in P'\cap C_{\n,u},
      x\in \Phi(t, P')} 
    \zeta_x> a_{\n}\delta) + \ex[\tN(0)\one{\tN(0)>n}] \, .
  \end{multline*}
  Since $\lim_{n\toi}\ex[\tN(0)\one{\tN(0)>n}]=0$ due to
  $\ex \tN(0)<\infty$, to show (\ref{eq:ac_moving_maxima_aux1}) it
  suffices to prove that
  \begin{align}\label{eq:ac_moving:maxima_aux2}
    \lim_{u\toi}\limsup_{\n \toi} \pr(\tmax_{t\in P'\cap
    C_{\n,u}, 
    x\in \Phi(t, P')} \zeta_x> a_{\n}\delta) = 0 \, .
  \end{align}
  This holds since $a_{\n}$ is chosen so that (\ref{eq:a_n})
  holds, while for all $u$,
  $|C_{\n,u}|/\n^d \leq |B_{b_{\n}c}|/\n^d = \mathrm{const} \cdot \,
  b_{\n}^d/\n^d \to 0$ as $\n\toi$ by the choice of
  $(b_\n)_{\n}$. Indeed, the refined Campbell's formula
  (\ref{eq:refined_Campbell}) gives that
  \begin{align*}
    \pr(\tmax_{t\in P'\cap C_{\n,u}, x\in \Phi(t, P')}
    \zeta_x > a_{\n}\delta)
    & \leq |B_{b_{\n} c}|
      \pr(\tmax_{x\in \tPhi(0)} \zeta_x >a_{\n}\delta) \\
    &= \frac{|B_{b_{\n} c}|}{\n^d} \frac{\pr(\tmax_{x\in
      \tPhi(0)} 
      \zeta_x >a_{\n}\delta)}{\pr(\xi > a_{\n}\delta)}  n^d \pr(\xi > a_{\n}\delta) \, ,
  \end{align*}
  for all $u,\n>0$. Due to (\ref{eq:a_n_eps}),
  (\ref{eq:mov_max_xi_is_RV}) and
  (\ref{eq:mov_max_xi_RV_special_case}), $b_{\n}/\n\to 0$ implies that
  (\ref{eq:ac_moving:maxima_aux2}) holds.
\end{proof}


\begin{proposition}
{If \Cref{hypo:stabilization} holds,}  \Cref{hypo:AI_of_extr_blocks} is satisfied for any $(l_{\n})_{\n}$
  such that $l_{\n}\to \infty$ and $l_{\n}/b_{\n}\to 0$ as $\n\toi$.
\end{proposition} 
\begin{proof}
  Recall that for a fixed family $(l_{\n})_{\n}$, and each $\n>0$ and
  $\bi \in I_{\n}$, the block of indices $J_{\n,\bi}$ is defined by
  (\ref{eq:blocks_of_indices}), its trimmed version $\hat{J}_{\n,\bi}$
  by (\ref{eq:trimmed_blocks}), and that
  $\widehat{\bB}_{\n,\bi}:= X_{\hat{J}_{\n,\bi}}$. Let now
  $f_{\n, \bi} \in \sF$, $\n>0, \bi \in I_{\n}$, be an arbitrary
  family of functions which satisfy \eqref{eq:fmudelta} for a same
  $\delta>0$. For notational convenience, write
  \begin{align*}
    W_{\n,\bi} =  \exp\left\{-f_{\n,\bi}(T_{1,a_{\n}}
    \widehat{\bB}_{\n,\bi})\right\}
  \end{align*}
  for all $\n>0$, $\bi \in I_{\n}$. To confirm
  \Cref{hypo:AI_of_extr_blocks}, we need to show that
  \begin{align}\label{eq:ai_aux1}
    \lim_{\n\toi}\big| \ex\big[\tprod_{\bi \in I_{\n}} W_{\n,\bi}\big]
    -  \tprod_{\bi \in I_{\n}}\ex\big[ W_{\n,\bi}\big]\big|=0 \, .
  \end{align}

  We first extend the definition of the radius $R$ from \Cref{hypo:stabilization} by setting $R(t, \mu'):=R(\shift{t}\mu')$, for all $\mu'\in \tsN_g$, $t\in \mu'$. Using shift-equivariance of $\Phi$, Campbell's formula and the assumption $\pr(R(\tP')<\infty)=1$, it is not difficult to show that almost surely, for all $t\in P'$, $\Phi(t,P') = \Phi(t,\nu')$ for all $\nu'$ which coincide with $P'$ on $B_{R(t,\mu')}(t)$.
  
 Thus, for every $t\in P'$, $R(t,P')<u$ implies
    \begin{align*}
        X(t)=\psi(t,P)= \tmax_{x\in \Phi(t, P')} w(x-t) \zeta_x = \psi(t,P_{B_{u}(t)}) \, .
    \end{align*}
Furthermore,  the
value of $W_{\n, \bi}$ depends
only on those $t\in P'\cap \hat{J}_{\n,\bi}$ with score
$\psi(t,P)>a_{\n}\delta$. In particular, if
  \begin{align*}
      \max_{t\in \hat{J}_{\n,\bi}\cap P'}
    R(t,P')\one{\psi(t,P)>a_{\n}\delta} <l_{\n} \, ,
  \end{align*}
  the random variable $W_{\n,\bi}$ depends only on $P$ restricted to $J_{\n,\bi}$.
  
Construct now $k_{\n}^d$ i.i.d.\ Poisson processes $P_{(\bi)}, \bi\in I_{\n}$, with common distribution equal to the distribution of $P$ and such that for each $\bi \in I_{\n}$, restrictions of $P_{(\bi)}$ and $P$ on the block $J_{\n, \bi}$ coincide. Furthermore, let for each $\bi \in I_{\n}$,  $W_{\n,\bi}^{*}$ be constructed from $P_{(\bi)}$ in the same way as $W_{\n,\bi}$ is constructed from $P$.
In particular, since $W_{\n,\bi}^{*}$'s are independent, 
\begin{align*}
    \big| \ex\big[\tprod_{\bi \in I_{\n}} W_{\n,\bi}\big]
    -  \tprod_{\bi \in I_{\n}}\ex\big[ W_{\n,\bi}\big]\big| = \big| \ex\big[\tprod_{\bi \in I_{\n}} W_{\n,\bi} - \tprod_{\bi \in I_{\n}} W_{\n,\bi}^{*}\big] \big| \, ,
\end{align*}
and moreover, one has $W_{\n,\bi} = W_{\n,\bi}^*$ whenever
\begin{align*}
    \max_{t\in \hat{J}_{\n,\bi}\cap P'}
    R(t,P')\one{\psi(t,P)>a_{\n}\delta} <l_{\n} \; \text{ and } \; \max_{t\in \hat{J}_{\n,\bi}\cap P'}
    R(t,P_{(\bi)}')\one{\psi(t,P_{(\bi)})>a_{\n}\delta} <l_{\n} \, .
\end{align*}
Thus, since  $\cup_{\bi\in I_{\n}} \hat{J}_{\n,\bi}\subseteq [0,\n]^d$ and $0\leq W_{\n,\bi}\le 1$,
\begin{align*}
    \big| \ex\big[\tprod_{\bi \in I_{\n}} W_{\n,\bi} - \tprod_{\bi \in I_{\n}} W_{\n,\bi}^{*}\big] \big|
    &\leq 2\pr\left(\max_{t\in [0,\n]^d\cap P'}
    R(t,P')\one{\psi(t,P)>a_{\n}\delta} \ge l_{\n}\right) \\
    & \leq 2 \ex\big[\tsum_{t\in P'\cap [0,\n]^d}
      \ind{ R(t,P') \geq l_{\n}, \psi(t,P)>a_{\n}\delta}\big]
  \end{align*}
  Using
  shift-invariance and the refined Campbell's theorem
  (\ref{eq:refined_Campbell}) we obtain
  \begin{align*}
    \big| \ex\big[\tprod_{\bi \in I_{\n}} W_{\n,\bi}\big]
    -  \tprod_{\bi \in I_{\n}}\ex\big[ W_{\n,\bi}\big]\big|
    & \leq 2 \ex\big[\tsum_{t\in P'\cap [0,\n]^d}
      \ind{ R(0,\shift{t}P') \geq l_{\n}, \psi(0,\shift{t}P)>a_{\n}\delta}\big] \\
    &= 2 \n^d \pr(R(0,\tP')\geq l_{\n}, \psi(0,\tP)>a_{\n}\delta) \, ,
  \end{align*}
  for all $\n>0$. Since
  \begin{displaymath}
    \xi=\psi(0,\tP)\leq w_{*} \max_{x\in \tPhi(0)} \zeta_{x}
  \end{displaymath}
  for $w_{*}=\max_{t\in \R^d} w(t)\in [1,\infty)$ and $\tP'$ is
  independent of the $\zeta$'s, conditioning on $\tP'$ yields that
  \begin{align*}
    2\n^d \pr(R(0,\tP')\geq l_{\n}, \psi(0,\tP)>a_{\n}\delta)
    \leq 2\n^d \ex[\one{R(0,\tP')\geq l_{\n}}\tN(0)] \pr(\zeta>a_{\n}\delta/w_{*}) \, .
  \end{align*}
  Due to (\ref{eq:mov_max_xi_is_RV}) and (\ref{eq:a_n_eps}),
  $\n^d\pr(\zeta>a_{\n}\delta/w_{*})$ converges to a positive
  constant, so (\ref{eq:ai_aux1}) follows by dominated convergence
  since $l_{\n}\toi$, $\pr(R(0,\tP')<\infty)=1$ and
  $\ex \tN(0)<\infty$.

\end{proof}

The above arguments lead to the following conclusion. 

\begin{proposition}
  Let $X$ be defined as in the beginning of
  \Cref{sub:moving_maxima_tail}. Assume $w$ is a bounded function,
  that $\ex [\tN(0)]<\infty$, and that there exists an $R$ satisfying
  \Cref{hypo:stabilization}. Let $(a_{\n})_{\n}$ be as in
  (\ref{eq:a_n}). Then any family $(b_{\n})_{\n}$, such that
  $b_{\n}\to \infty$ and $b_{\n}/\n \to 0$ as $\n \toi$, satisfies all
  of the assumptions of \Cref{thm:main_result} (with $\alpha$ being
  equal to the tail index of $\zeta$ and $r\equiv 1$) and, therefore,
  the convergence of the extremal blocks in (\ref{eq:pp_convergence})
  holds.
\end{proposition}

\section{Proof of \Cref{thm:main_result}}
\label{subs:proof_main_result}

Recall that $X$ is a stationary marked point process
on $\E$ which admits a tail configuration $Y$ with tail index
$\alpha>0$ and a scaling function $r$ of scaling index
$\beta\in \R$. Moreover, assume that $(a_{\n})_{\n>0}$ satisfies
(\ref{eq:a_n}) and fix a family of block side lengths
$(b_{\n})_{\n>0}$.

For notational convenience, for all $\n>0$ and $y>0$ denote in the
sequel
\begin{align*}
  X_{\n}:=X_{[0,b_\n]^d}, \; \text{and } \; T^{\n}_{y} := T_{r(a_\n
  y), a_\n y}.
\end{align*}
We first extend (\ref{eq:tail_process}) to convergence in the space
$\sN_{01}$ with the $\borel_{01}$-vague topology.

\begin{proposition}\label{prop:conv_to_tail_in_N_star}
  Assume that $(b_\n)_{\n>0}$ satisfies Assumptions~\ref{hypo:t_n} and
  \ref{hypo:AC}. Then 
  \begin{align}\label{eq:Y_in_lo}
    \pr(Y\in \sN_{01})=1 .
  \end{align}
  Furthermore, for all $y>0$
  \begin{align}\label{eq:conv_to_tail_in_N_star}
    \pr(T^{\n}_{y} (\tX_{D_{\tau}}) \in \cdot \mid \xi>a_\n y)
    \wto \pr(Y\in \cdot\,) \quad\text{as}\; \n\toi
  \end{align}
  on $\sN_{01}$, where $(D_\n)_{\n>0}$ is any family of subsets of
  $\R^d$ such that 
  \begin{align}\label{eq:conv_to_tail_in_N_star_incr_sets}
    \ball_{c_1 b_\n}\subseteq D_\n \subseteq \ball_{c_2 b_\n},\quad \n>0,
  \end{align}
  for some constants $0<c_1< c_2$.
\end{proposition}
\begin{proof}
  Recall that $\pr(Y\in \sN_{11})=1$, i.e., $Y$ a.s. has finitely many
  points in $B\times (\delta,\infty)$ for all bounded $B\subset\R^d$
  and all $\delta>0$. To prove (\ref{eq:Y_in_lo}), we need to show
  that $Y(\R^d\times (\delta,\infty))<\infty$ a.s.\ for all
  $\delta>0$.  Fix $u,\delta>0$ and observe that
  \begin{align*}
    \pr\Big(\sup_{ |t|\geq u} Y(t) \ge \delta\Big)
    = \lim_{u'\toi} \pr\Big(\max_{ u'\ge |t|\geq u}Y(t)> \delta\Big) \, .
  \end{align*}
  By definition (\ref{eq:tail_process}) of $Y$, for all but at
  most countably many $u'>u$,
  \begin{align*}
    \pr\Big(\max_{ u'\ge |t|> u}Y(t)\ge \delta\Big)
    = \lim_{\n \toi} \pr\Big(\max_{  u'\ge |t|/r(a_{\n})\geq u}\tX(t)
    > a_{\n} \delta\mid \tX(0)> a_{\n}\Big) \, .
  \end{align*}
  Since $b_{\n}/r(a_{\n})\to \infty$ by \Cref{hypo:t_n}, the two
  previous relations imply that
  \begin{align*}
    \pr\Big(\sup_{  |t|\geq u} Y(t) \ge \delta\Big)
    \leq \limsup_{\n \toi}
    \pr\Big(\max_{\, b_{\n}/r(a_{\n})\ge |t|/r(a_{\n})\geq u}\tX(t)>
    a_{\n} \delta
    \mid \tX(0)\ge a_{\n}\Big) \, .
  \end{align*}
  \Cref{hypo:AC} (with $c=\eps=1$) yields
  \begin{align*}
    \lim_{u\toi} \pr\Big(\sup_{ |t|\geq u} Y(t) \ge \delta\Big) = 0 \, .
  \end{align*}
  Thus, 
  \begin{align}\label{eq:Y_in_lo_aux_1}
    \pr\Big(\bigcup_{u>0} \bigcap_{|t|>u} \{Y(t)< \delta\}\Big) = 1 \, ,
  \end{align}
  and since $\pr(Y\in \sN_{11})=1$, one has that
  $Y(\R^d\times [\delta,\infty))<\infty$ a.s. Since $\delta$
  was arbitrary, this proves (\ref{eq:Y_in_lo}).

  We now turn to (\ref{eq:conv_to_tail_in_N_star}). Fix a
  $y>0$. For all but at most countably many $u>0$, the definition of
  the tail configuration implies that 
  \begin{align}\label{eq:conv_to_tail_in_N_star_aux_1}
    \pr \left((T^{\n}_y \tX)_{\ball_u} \in \cdot \mid \tX(0)> a_{\n} y
    \right) 
    \wto \pr(Y_{\ball_u} \in \cdot \,) \quad \text{as}\; \n\toi 
  \end{align}
  in $\sN_{11}$ with the $\borel_{11}$-vague topology.  For measures
  in $\sN_{11}$ whose support is in $\ball_u\times (0,\infty)$,
  $\borel_{11}$-vague topology is actually equivalent to the (in
  general stronger) $\borel_{01}$-vague topology. Thus,
  (\ref{eq:conv_to_tail_in_N_star_aux_1}) holds on $\sN_{01}$ with
  respect to the $\borel_{01}$-vague topology as well. Furthermore, it
  is easy to see that, due to (\ref{eq:Y_in_lo_aux_1}),
  $Y_{\ball_u} \to Y$ almost surely in $\sN_{11}$ as $u\toi$. In
  particular,
  \begin{align}\label{eq:conv_to_tail_in_N_star_aux_2}
    \pr(Y_{\ball_u} \in \cdot \,) \wto \pr(Y\in \cdot \, ) \quad \text{as}\; u\toi
  \end{align}
  on $\sN_{11}$ with respect to the $\borel_{01}$-vague topology. By
  the classical result on weak convergence of probability measures
  (see \cite[Theorem~4.2]{billingsley:1968}), to prove
  (\ref{eq:conv_to_tail_in_N_star}) it suffices to show that
  \begin{align}\label{eq:conv_to_tail_in_N_star_aux_3}
    \lim_{u\toi} \limsup_{\n\toi} \ex\Big[\metric((T^{\n}_y
    \tX)_{\ball_u}, 
    T^{\n}_{y} (\tX_{D_{\tau}}) ) \mid \tX(0)>a_{\n} y\Big] = 0 \, ,
  \end{align}
  where $\metric$ is the metric from (\ref{eq:metric_on_l0}) (which
  generates the $\borel_{01}$-vague topology on $\sN_{01}$).
  For this, observe that 
  \begin{align*}
    (T^{\n}_y \tX)_{\ball_u} = T^{\n}_y (\tX_{\ball_{r(a_{\n} y) u}}) \, .
  \end{align*}
  Let $u,\epsilon>0$ be arbitrary. 
  Due to the first inclusion in
  (\ref{eq:conv_to_tail_in_N_star_incr_sets}) and since
  $b_{\n}/r(a_{\n} y)\to \infty$ by \Cref{hypo:t_n},
  $\ball_{r(a_{\n} y)u} \subseteq D_{\tau}$ for all sufficiently large
  $\n$.  Observe that if
  \begin{equation}
    \label{eq:2a}
    \tmax_{t\in D_\n \setminus \ball_{r(a_{\n} y)u}} \tX(t)
    \leq a_{\n} y \epsilon \, ,
  \end{equation}
  then $T^{\n}_y (\tX_{\ball_{r(a_{\n} y) u}})$ and
  $T^{\n}_{y} (\tX_{D_{\tau}})$ coincide when restricted to
  $\R^d\times (1/r,\infty)$ with $r < 1/\epsilon$. Since $\metric_0$
  is bounded by 1, 
  \begin{align}\label{eq:useful_trick}
    \metric(T^{\n}_y (\tX_{\ball_{r(a_{\n} y) u}}), T^{\n}_{y}
    (\tX_{D_{\tau}}) ) 
    \leq \int_{1/\epsilon}^{\infty} e^{-r} \dx r = e^{-1/\epsilon} \, .
  \end{align}
  If \eqref{eq:2a} does not hold, we use the fact that $\metric$ is
  bounded by 1.  Thus,
  \begin{align*}
    \limsup_{\n \toi}
    &\; \ex\Big[\metric(T^{\n}_y (\tX_{\ball_{r(a_{\n} y)
      u}}), 
      T^{\n}_{y} (\tX_{D_{\tau}}) ) \mid \tX(0)>a_{\n} y\Big] \\
    &\leq  e^{-1/\epsilon}
      + \limsup_{\n \toi}\pr\Big(\tmax_{t\in D_\n \setminus
      \ball_{r(a_{\n} y)u}}
      \tX(t) > a_{\n} y \epsilon \mid \tX(0)>a_{\n} y\Big) \\
    &\leq e^{-1/\epsilon}
      + \limsup_{\n \toi}\pr\Big(\tmax_{t\in \ball_{b_{\n} c_2}
      \setminus \ball_{r(a_{\n} y)u}}\tX(t) > a_{\n} y
      \epsilon \mid \tX(0)>a_{\n} y\Big) \, ,
  \end{align*}
  where the last inequality holds, since
  $D_{\n}\subseteq \ball_{b_{\n} c_2}$. If now one lets $u\toi$ and
  then $\epsilon\to 0$, \Cref{hypo:AC} implies
  (\ref{eq:conv_to_tail_in_N_star_aux_3}), and this proves
  (\ref{eq:conv_to_tail_in_N_star}).
\end{proof}

For all $\mu\in \sN_{01}$ and $y>0$, let 
\begin{align*}
  A_y(\mu):=\min\{t\in \R^d : (t,s)\in \mu, s>y\},
\end{align*}
where the minimum is taken with respect to the lexicographic order; if
$M(\mu)\leq y$, set $A_y(\mu):=0$. Note that $A_y$ is well defined,
since for every $\mu\in \sN_{01}$ and $y>0$ there are at most
finitely many $(t,s)\in \mu$ with $s>y$. Observe also that $A_y$ is
equivariant under translations, that is, if $M(\mu)>y$,
\begin{align} \label{eq:covariance_A}
  A_y(\shift{z}\mu)=A_y(\mu)-z,\quad  z \in \R^d. 
\end{align}
In particular, $A_1$ is precisely the first exceedance anchoring
function $A^{\mathrm{fe}}$ from (\ref{eq:first_exc_anchor}). As shown
in \Cref{lem:theta_positive}, $\pr(A_1(Y)=0)$ is positive
whenever $\pr(Y\in \sN_{01})=1$, which holds, for instance, under
Assumptions~\ref{hypo:t_n} and \ref{hypo:AC}. Recall also that $X_{\n}:=X_{[0,b_\n]^d}$.

\begin{proposition}\label{prop:conv_of_clusters}
  Assume that Assumptions~\ref{hypo:t_n} and \ref{hypo:AC} hold.  Then
  for every $y>0$,
  \begin{align}\label{eq:theta}
    \frac{\pr(M(X_{\n})>a_\n y)}{b_\n^d \pr(\xi>a_\n y)} \to
    \pr(A_1(Y)=0)
    \quad \text{as}\; \n\toi, 
  \end{align}
  and
  \begin{align}\label{eq:conv_of_clusters}
    \pr\Big([T^{\n}_{y} (X_{\n})] \in \cdot \mid M(X_{\n})>a_\n y
    \Big) 
    \wto \pr\Big([Y] \in \cdot \mid  A_1(Y)=0\Big) \;\;\;\; \text{on }   \tsN_{01} . 
  \end{align}
\end{proposition}
\begin{proof}
  Consider a function $h:\sN_{01}\to [0,\infty)$, which is bounded
  continuous and shift-invariant.  First, observe that we can decompose
  \begin{multline*}
    h\big(T^{\n}_{y} (X_{\n})\big) \ind{M(X_{\n})>a_\n y} \\ = \tsum_{(t,s) \in X} h\big(T^{\n}_{y} (X_{\n})\big) \ind{t\in
    [0,b_{\n}]^d, 
    A_{a_{\n} y}(X_{\n})=t, X(t)>a_{\n} y}\, .
  \end{multline*}
  Since $h$ is shift-invariant and since $T^{\n}_{y}$ scales the
  positions with $r(a_{\n}y)^{-1}$, 
  \begin{align*}
    h\big(T^{\n}_{y} (X_{\n})\big) = h\big(\shift{t/r(a_{\n}y)}T^{\n}_{y}
    (X_{\n})\big) 
    = h\big(T^{\n}_{y} (\shift{t} X_{\n})\big)
    =  h\big(T^{\n}_{y} ((\shift{t}X)_{[0,b_{\n}]^{d}-t})\big) \, .
  \end{align*}
  By the definition of $A_{a_{\n} y}$,
  \begin{displaymath}
    \big\{A_{a_{\n} y}(X_{\n})=t\big\} = \big\{A_{a_{\n} y}(\shift{t} X_{\n})=0\big\}
    = \big\{A_{a_{\n} y}((\shift{t}X)_{[0,b_{\n}]^{d}-t})=0\big\} \, .
  \end{displaymath} 
  Next, $\{X(t)>a_{\n} y\}=\{\shift{t}X(0)>a_{\n}y\}$.
  All of the above implies that 
  \begin{align*}
    \ex\Big[h(T^{\n}_{y} (X_{\n})) \one{M(X_{\n})>a_\n y}\Big]
    = \ex\Big[ \tsum_{(t,s) \in X} g(t, \shift{t}X)\Big] 
  \end{align*}
  for an obvious choice of a function
  $g:\R^d \times \sN_{01} \to [0,\infty)$. Refined Campbell's theorem
  (\ref{eq:refined_Campbell}) yields that
  \begin{multline*}
    \ex\Big[h(T^{\n}_{y} (X_{\n})) \one{M(X_{\n})>a_\n y}\Big]
    = \int_{\R^d} \ex[g(t,\tX)] \dx t \\
    = \int_{[0,b_{\n}]^d} \ex\Big[h(T^{\n}_{y}
    (\tX_{[0,b_{\n}]^{d}-t})) \ind{A_{a_{\n} y}
      (\tX_{[0,b_{\n}]^{d}-t})=0, \tX(0)>a_{\n}y}\Big] \dx t \\
    = b_{\n}^d \int_{[0,1]^d} \ex\Big[h(T^{\n}_{y}
    (\tX_{[0,b_{\n}]^{d}-b_{\n}s}))
    \ind{A_{1}(T^{\n}_{y}(\tX_{[0,b_{\n}]^{d}-b_{\n}s}))=0, \tX(0)>a_{\n}y}\Big] \dx s \, ,
  \end{multline*}
  where for the last equality we used the substitution $s=t/b_{\n}$
  and the fact that $A_{a_{\n} y}(\mu) = 0$ if and only if
  $A_{1}(T^{\n}_{y}\mu) = 0$. In particular (recall that
  $\xi=\tX(0)$),
  \begin{equation}\label{eq:vague_conv_of_cluster_aux1}
    \frac{\ex\Big[h(T^{\n}_{y} (X_{\n}))
      \one{M(X_{\n})>a_\n y}\Big]}{b_{\n}^d \pr(\xi>a_\n y)}
    = \int_{[0,1]^d} \ex\Big[f(T^{\n}_{y} (\tX_{[0,b_{\n}]^{d}-b_{\n}s}))
    \mid \tX(0)>a_{\n}y \Big] \,  \dx s \, ,
  \end{equation}
  where $f(\mu):=h(\mu)\one{A_1(\mu)=0}$.  Since for every fixed
  $s\in (0,1)^d$, sets $D_{\n}:=[0,b_{\n}]^{d}-b_{\n}s$, $\n>0$,
  satisfy (\ref{eq:conv_to_tail_in_N_star_incr_sets}),
  (\ref{eq:conv_to_tail_in_N_star}) implies that
  \begin{align}\label{eq:vague_conv_of_clusters_aux2}
    \lim_{\n\toi}\ex\big[f(T^{\n}_{y} (\tX_{[0,b_{\n}]^{d}-b_{\n}s}))
    \mid \tX(0)>a_{\n}y \big] = \ex[f(Y)]
    = \ex\big[h(Y)\one{A_1(Y)=0}\big] \, .
  \end{align}
  Observe here that $f$ is not continuous on the whole $\sN_{01}$,
  but since the probability that $(t,1)\in Y$ for some $t\in\R^d$ is
  zero (due to Propositions \ref{prop:S0t_is_Pareto} and \ref{prop:spectral_is_indep}), it is continuous on the support of $Y$ -- this justifies the
  use of (\ref{eq:conv_to_tail_in_N_star}). By the dominated
  convergence theorem and since the limit in
  (\ref{eq:vague_conv_of_clusters_aux2}) does not depend on $s$, 
  (\ref{eq:vague_conv_of_cluster_aux1}) yields that
  \begin{align}\label{eq:vague_conv_of_clusters}
    \lim_{\tau\to\infty}\frac{\ex\big[h(T^{\n}_{y} (X_{\n}))
    \one{M(X_{\n})>a_\n y}\big]}{b_{\n}^d \pr(\xi>a_\n y)}
    = \ex\big[h(Y)\one{A_1(Y)=0}\big] \, .
  \end{align}
  Convergence in (\ref{eq:theta}) follows from
  (\ref{eq:vague_conv_of_clusters}) with $h\equiv 1$.  To prove
  (\ref{eq:conv_of_clusters}), we notice that
  \begin{align*}
    \lim_{\n\toi} \ex\Big[h(T^{\n}_{y} (X_{\n}))
    \mid M(X_{\n})>a_\n y \Big] = \ex\big[h(Y) \mid  A_1(Y)=0\big] 
  \end{align*}
  follows directly from (\ref{eq:theta}) and
  (\ref{eq:vague_conv_of_clusters}).
\end{proof}

Let $\bM^{*}$ be the space of all Borel measures on $\tsN_{01}^{*}$ which
are finite on all sets from the family
\begin{align*}
  \borel^{*}=\big\{B\subseteq \tsN_{01}^{*} :
  \exists\eps>0, \forall [\mu]\in B,  M([\mu])>\eps\big\} \, .
\end{align*}
Equip $\bM^{*}$ with the vague topology generated by $\borel^*$.
Furthermore, let $\bM$ be the space of all Borel measures on
$[0,1]^d\times \tsN_{01}^*$ taking finite values on $[0,1]^d\times B$
for all $B\in \borel^*$, and equip it with the corresponding vague
topology. Observe that the intensity measure
\begin{align*}
  \ex\big[N_{\n}(\, \cdot \,)\big]
  = \tsum_{\bi \in I_{\n}} \pr\big( \, (\bi b_{\n}/{\n}, T^{\n}_1
  [\bB_{\n, \bi}]) \in \cdot \,\big)
\end{align*}
of the point process $N_{\n}$ from (\ref{eq:PP_blocks}), is an element
of $\bM$ for all $\n>0$. Recall that $\bN\subset \bM$ defined right before \Cref{thm:main_result} is the subset of all \textit{counting} measures.

\begin{proposition}[Intensity
  convergence] \label{prop:intensity_convergence} Assume that
  Assumptions~\ref{hypo:t_n} and \ref{hypo:AC} hold. Then
  \begin{align}\label{eq:intensity_conv_1}
    k_{\n}^d \;\pr\Big([T^{\n}_1 X_{\n}] \in \cdot \, \Big) \vto \nu
    \quad \text{in}\; \bM^{*}\; \text{as}\; \n\toi\; ,
  \end{align}
  where
  \begin{align}\label{eq:nu}
    \nu(\, \cdot \,) = \vartheta \int_{0}^\infty
    \pr\big([T_{u^{-\beta}, 
    u^{-1}} Q] \in \cdot \,\big) \alpha u^{-\alpha-1} \dx u,
  \end{align}
  and $Q\in \sN_{01}$ has distribution
  (\ref{eq:normalized_typical_cluster_cond}). In particular,
  \begin{align}\label{eq:intensity_conv_2}
    \ex\big[N_{\n}(\, \cdot \,)\big] \vto \Leb\times \nu(\cdot)\quad
    \text{as}\; \n\toi,
  \end{align}
 in $\bN$,  where $\Leb$ is the Lebesgue measure on $\R^d$.
\end{proposition}
\begin{proof}
  Let $h:\bM^{*}\to [0,\infty)$ be a bounded and continuous function such
  that for some $\epsilon>0$, $h([\mu])=0$ whenever
  $M([\mu])\leq \eps$. Then
  \begin{equation}
    \label{eq:11}
    k_{\n}^{d} \;\ex\big[h([T_{1}^{\n}X_{\n}])\big]
    = k_{\n}^{d} \;\ex\Big[h([T_{1}^{\n}X_{\n}])\one{M(X_{\n})>a_{\n}\epsilon}\Big].
  \end{equation}
  Since $k_{\n}\sim \n / b_{\n}$ as $\n\toi$, (\ref{eq:a_n}) implies
  that $k_{\n} \sim (b_{\n}^d \pr(\xi >a_{\n}))^{-1}$ as $\n \toi$.
  Thus, as $\n\toi$,
  \begin{multline}\label{eq:intensity_conv_aux1}
    k_{\n}^{d} \ex\Big[h([T_{1}^{\n}X_{\n}])\one{M(X_{\n})>a_{\n}\epsilon}\Big] \\
    \sim \ex\big[h([T_{1}^{\n}X_{\n}])\mid
    M(X_{\n})>a_{\n}\epsilon\big] 
    \; \frac{\pr(M(X_{\n})>a_{\n}\epsilon)}{b_{\n}^d
      \pr(\xi >a_{\n}\epsilon)} \; \frac{\pr(\xi >a_{\n}\epsilon)}{\pr(\xi >a_{\n})} \, .
  \end{multline}
  By (\ref{eq:theta}), the second term on the right-hand side of
  (\ref{eq:intensity_conv_aux1}) converges to $\pr(A_1(Y)=0)$, which
  in turn by \Cref{prop:anchors_are_equivalent} equals
  $\vartheta=\pr(A(Y)=0)$. By regular variation property
  (\ref{eq:S_0_is_RV}), the third term tends to
  $\epsilon^{-\alpha}$. For the first term, recall that $T_{y}^{\n}$
  scales the positions with $r(a_{\n}y)^{-1}$ and scores with
  $(a_{\n}y)^{-1}$. Thus,
  \begin{align*}
    T_{1}^{\n}X_{\n} = T_{r(a_{\n})/r(a_{\n}\epsilon), \epsilon^{-1}}T_{\epsilon}^{\n}X_{\n}.
  \end{align*}
  By assumption, $r$ is a regularly varying function with index
  $\beta$, so $r(a_{\n})/r(a_{\n}\epsilon)\to \epsilon^{-\beta}$ as
  $\n\toi$. In particular, (\ref{eq:conv_of_clusters}) and the
  extended continuous mapping theorem (see
  \cite[Theorem~5.5]{billingsley:1968}) imply that
  \begin{align*}
    \lim_{\n\toi} \ex\big[h([T_{1}^{\n}X_{\n}])\mid
    M(X_{\n})>a_{\n}\epsilon\big]
    &= \lim_{\n\toi} \ex\big[h(T_{r(a_{\n})/r(a_{\n}\epsilon),
      \epsilon^{-1}}[T_{\epsilon}^{\n}X_{\n}])\mid M(X_{\n})>a_{\n}\epsilon\big]\\
    &= \ex\big[h(T_{\epsilon^{-\beta}, \epsilon^{-1}}[Y]) \mid A_1(Y)=0\big].
  \end{align*}
  By \Cref{prop:anchors_are_equivalent}
  and (\ref{eq:spectral_decomp_of_Z}), we can rewrite the limit as
  \begin{align*}
    \ex[h(T_{\epsilon^{-\beta}, \epsilon^{-1}}[Y]) \mid A_1(Y)=0]
    & = \ex[h(T_{\epsilon^{-\beta}, \epsilon^{-1}}[Y]) \mid A^{\mathrm{fm}}(Y)=0] \\
    & = \int_{1}^{\infty} \ex\big[h(T_{\epsilon^{-\beta},
      \epsilon^{-1}}[T_{u^{-\beta}, u^{-1}} Q]) \big] \alpha u^{-\alpha-1} \dx u \\
    & = \int_{1}^{\infty} \ex\big[h([T_{(\epsilon u)^{-\beta},
      (\epsilon u)^{-1}} Q]) \big] \alpha u^{-\alpha-1} \dx u \\
    & = \epsilon^{\alpha} \int_{\epsilon}^{\infty}
      \ex\big[h([T_{y^{-\beta}, y^{-1}} Q]) \big] \alpha y^{-\alpha-1} \dx y \\
    & = \epsilon^{\alpha} \int_{0}^{\infty}
      \ex\big[h([T_{y^{-\beta}, y^{-1}} Q]) \big] \alpha y^{-\alpha-1} \dx y \, .
  \end{align*}
  In the above we used the substitution $y=u\epsilon$ to obtain the
  fourth equality. For the final equality, note that
  $M(T_{y^{-\beta}, y^{-1}} Q)=y$ since $M(Q)=1$ a.s. In
  particular, $h([T_{y^{-\beta}, y^{-1}} Q])= 0$ a.s.\ whenever
  $y\le\epsilon$ by the properties of $h$ stated in the beginning of
  the proof.

  Bringing everything together, \eqref{eq:11} and
  (\ref{eq:intensity_conv_aux1}) imply that
  \begin{align*}
    \lim_{\n \toi} k_{\n}^{d} \;\ex\big[h([T_{1}^{\n}X_{\n}])\big]
    &=\lim_{\n \toi} k_{\n}^{d} \;\ex\big[h([T_{1}^{\n}X_{\n}])\ind{M(X_{\n})>a_{\n}\epsilon}\big]\\
    & =\epsilon^{\alpha} \int_{0}^{\infty}
      \ex\big[h([T_{y^{-\beta}, y^{-1}} Q]) \big] \alpha y^{-\alpha-1}\dx y 
      \,  \vartheta \, \epsilon^{-\alpha} \\
    &= \vartheta \int_{0}^{\infty}  \ex\big[h([T_{y^{-\beta}, y^{-1}}
      Q])\big] \alpha y^{-\alpha-1} \dx y = \nu(h) \, , 
  \end{align*}
  where $\nu$ is defined at (\ref{eq:nu}). Since $h$ is arbitrary,
  this proves (\ref{eq:intensity_conv_1}).

  We now prove (\ref{eq:intensity_conv_2}). By
  \cite[Lemma~4.1]{kallenberg:2017}, it suffices to prove that
  $\ex\big[N_{\n}(g)\big] \to (\Leb\times \nu)(g)$ for all
  $g:[0,1]^d\times \tsN_{01}^{*}\to [0,\infty)$ of the form
  $g(t, [\mu]) = \bone_{(a,b]}(t) \bone_{A}([\mu])$, where $(a,b]$ is
  the parallelepiped in $\R^d$ determined by
  $a=(a_1,\dots, a_d), b=(b_1,\dots,b_d)\in [0,1]^d$ with
  $a_j\leq b_j$ for all $j$, and $A\in \borel^{*}$ (with $\borel^{*}$
  defined just before \Cref{prop:intensity_convergence}) such that
  $\nu(\partial A)=0$.

  Recall that $X_{\n}=X_{[0,b_{\n}]^d}$, so
  $X_{\n}=\bB_{\n,\boldsymbol{1}}$ and due to stationarity of $X$,
  $\pr([X_{\n}]\in \cdot \, )= \pr([\bB_{\n, \bi}]\in \cdot \, )$ for
  all $\bi\in I_{\n}=\{1,\dots,k_{\n}\}^d$, where
  $k_{\n}= \lfloor\n/b_{\n}\rfloor$. Thus, as $\n \toi$,
  \begin{align*}
    \ex\big[N_{\n}(g)\big]
    & = \tsum_{\bi \in I_{\n}} \one{\bi b_{\n}/{\n} \in (a,b]}  
      \pr([T_1^{\n} X_{\n}] \in A)
      \sim \left(\frac{\n}{b_{\n}}\right)^d \,
      \prod_{j=1}^d (b_j-a_j) \pr([T_1^{\n} X_{\n}] \in A) \\
    &\sim \prod_{j=1}^d (b_j-a_j) k_{\n}^d \pr([T_1^{\n} X_{\n}]
      \in A) 
      \to \prod_{j=1}^d (b_j-a_j) \nu(A) = \Leb\times \nu \, (g) \, ,
  \end{align*}
  where in the penultimate step we applied
  (\ref{eq:intensity_conv_1}).
\end{proof}

To finish the proof of \Cref{thm:main_result}, we need the following
technical lemma.

\begin{lemma}\label{lem:I_blocks_cond_for_original_blocks}
  Assume that $(b_\n)_{\n>0}$ is such that Assumptions~\ref{hypo:t_n} and
  \ref{hypo:AC} hold. Then, Assumption~\ref{hypo:AI_of_extr_blocks}
  implies
  \begin{align}\label{eq:AI_blocks_cond_for_original_blocks}
    \ex\Big[\tprod_{\bi \in I_{\n}}
    \exp\left\{-f(\bi b_{\n} / {\n}, T^{\n}_{1} [{\bB}_{\n,\bi}])\right\}\Big] -
    \tprod_{\bi\in I_{\n}}
    \ex\Big[ \exp\left\{-f\left(\bi b_{\n} / {\n},
    T^{\n}_{1}[{\bB}_{\n,\bi}]\right)\right\}\Big] \to 0
  \end{align}
  as $\n\toi$ for every $f: [0,1]^d \times \tsN_{01} \to [0,\infty)$
  such that
  \begin{itemize}
  \item[(i)] for some $\epsilon>0$, $M(\mu)\leq \epsilon$ implies
    $f(t, [\mu])=0$ for all $t\in [0,1]^d$;
  \item[(ii)] $f$ is Lipschitz, that is, for some $c>0$,
    \begin{align*}
      |f(t, [\mu])- f(s,[\nu])| \leq c
      \max\big\{|t-s|, \tilde{\metric}([\mu],[\nu]) \big\}
    \end{align*}   
    for all $t,s\in [0,1]^d$ and nontrivial measures $\mu, \nu \in \sN_{01}$.
  \end{itemize}  
\end{lemma}

Observe that, for each $\n>0$, the first term on the left-hand side of
(\ref{eq:AI_blocks_cond_for_original_blocks}) is the Laplace
functional $L_f(N_{\n})$ of the point process $N_{\n}$, while the
second term is the Laplace functional of the point process
\begin{align*}
  N_\n^* := \left\{\big(\bi b_{\n} / \n \, , \;
  T^{\n}_1[\bB_{\n,\bi}^*]\big): 
  \bi \in I_{\n}\right\},
\end{align*}
where the blocks $\bB_{\n,\bi}^*$, $\bi \in I_{\n}$, are independent,
and for each $\bi\in I_{\n}$, $\bB_{\n,\bi}^*$ has the same
distribution as the original block $\bB_{\n,\bi}$.

\begin{proof}[Proof of \Cref{lem:I_blocks_cond_for_original_blocks}]
  Let $f$ be an arbitrary function satisfying the assumptions of the
  lemma for some $\epsilon$ and $c$. Fix an arbitrary
  $\delta<\epsilon$ and define a function
  $f^{\delta}:[0,1]^d \times \sN_{01} \to [0,\infty)$ by
  $f^\delta(t,\mu):= f(t,[\mu^{\delta}])$, where $\mu^\delta$ is the
  restriction of $\mu \in \sN_{01}$ to $\R^d\times
  (\delta,\infty)$. For all $t\in [0,1]^d$ and all
  $\mu\in \sN_{01}$, the Lipschitz property of $f$ implies that
  \begin{align}\label{eq:lipschitzness}
    |f(t,[\mu])- f^\delta(t,\mu)| \leq c
    \tilde{\metric}([\mu],[\mu^{\delta}]) 
    \leq c \metric(\mu, \mu^{\delta}) \leq c e^{-1/\delta} \, ,
  \end{align} 
  see (\ref{eq:useful_trick}) for a similar argument which justifies
  the last inequality.

  By the construction of $f^\delta$ and properties of
  $f$, \Cref{hypo:AI_of_extr_blocks} implies that
  (\ref{eq:AI_blocks_cond}) holds for the family
  \begin{align*}
    f_{\n,\bi}(\mu):= f^{\delta}(\bi b_{\n}/\n, \mu), \quad
    \n>0,\, \bi\in I_{\n}, \,\mu \in \sN_{01}.
  \end{align*}

  We first show that in (\ref{eq:AI_blocks_cond}) one can replace the
  trimmed blocks $\hat{\bB}_{\n,\bi}$ with the original ones
  $\bB_{\n,\bi}$. For this, observe that
  \begin{multline*}
    \Big|\ex\Big[\tprod_{\bi \in I_{\n}}
    \exp\left\{-f_{\n,\bi}(T^{\n}_{1} {\bB}_{\n,\bi})\right\}\Big]
    - \ex\Big[\tprod_{\bi \in I_{\n}}
    \exp\{-f_{\n,\bi}( T^{\n}_{1} \hat{\bB}_{\n,\bi})\}\Big]\Big| \\
    \leq \tsum_{\bi \in I_{\n}} \ex \Big|\exp\{-f_{\n,\bi}(T^{\n}_{1}
      {\bB}_{\n,\bi})\} - \exp\{-f_{\n,\bi}( T^{\n}_{1} \hat{\bB}_{\n,\bi})\}\Big| \, ,
  \end{multline*}
  where we used the elementary inequality
  \begin{align}\label{eq:elem_ineq}
    |\tprod_{i=1}^k a_i -  \tprod_{i=1}^k b_i| \leq \tsum_{i=1}^k |a_i-b_i|
  \end{align}
  valid for all $k$ and all $a_i,b_i\in[0,1]$,
  $i=1,\dots, k$. Recall the blocks of indices $J_{\n,\bi}$ and
  $\hat{J}_{\n,\bi}$ from (\ref{eq:blocks_of_indices}) and
  (\ref{eq:trimmed_blocks}), respectively. Observe that
  $M(X_{J_{\n,\bi}\setminus \hat{J}_{\n,\bi}})\leq a_{\n} \delta$
  implies that
  $f_{\n,\bi}(T^{\n}_{1} {\bB}_{\n,\bi})=f_{\n,\bi}( T^{\n}_{1}
  \hat{\bB}_{\n,\bi})$. In particular, due to stationarity of $X$,
  \begin{align*}
    \tsum_{\bi \in I_{\n}} \ex &\Big|\exp{\{-f_{\n,\bi}(T^{\n}_{1}
    {\bB}_{\n,\bi})\}} 
    - \exp\{-f_{\n,\bi}( T^{\n}_{1} \hat{\bB}_{\n,\bi})\}\Big| \\
    & \leq k_{\n}^d \;\pr(M(X_{[0,b_{\n}]^d\setminus [l_{\n},
      b_{\n}-l_{\n}]^d })> a_{\n} \delta) \\
    &\le k_{\n}^d \;\ex\Big[\tsum_{(t,s)\in X} \ind{t\in [0,b_{\n}]^d
      \setminus [l_{\n}, b_{\n}-l_{\n}]^d, s>a_{\n}\delta}\Big] \\
    &= k_{\n}^d \Leb([0,b_{\n}]^d\setminus [l_{\n}, b_{\n}-l_{\n}]^d)
      \;\pr(\tX(0)>a_{\n} \delta)\\ 
    &\sim \text{const } k_{\n}^d b_{\n}^{d-1} l_{\n}
      \pr(\tX(0)>a_{\n} \delta)
      \sim \text{const } \frac{l_{\n}}{b_{\n}} \n^{d}
      \pr(\tX(0)>a_{\n} \delta) 
      \to \text{const } \cdot 0\cdot \delta^{-\alpha}=0
  \end{align*}
  as $\n\toi$. In the third step we used the refined Campbell's
  formula, in the fourth the assumption $l_{\n}/b_{\n}\to 0$, in the
  fifth the fact that $k_{\n} \sim \n / b_{\n}$, and, finally,
  (\ref{eq:a_n_eps}). Thus,
  \begin{align*}
    \Big|   \ex\Big[\tprod_{\bi \in I_{\n}}
    \exp\left\{-f_{\n,\bi}(T^{\n}_{1} {\bB}_{\n,\bi})\right\}\Big]
    - \ex\Big[\tprod_{\bi \in I_{\n}}
    \exp\{-f_{\n,\bi}( T^{\n}_{1} \hat{\bB}_{\n,\bi})\}\Big]\Big|
    \to 0 \quad \text{as}\; \n \toi \, .
  \end{align*}
  After applying (\ref{eq:elem_ineq}), the same arguments also imply
  \begin{align*}
    \Big|\tprod_{\bi \in I_{\n}}
    \ex\Big[\exp\left\{-f_{\n,\bi}(T^{\n}_{1}
    {\bB}_{\n,\bi})\right\}\Big] 
    - \tprod_{\bi \in I_{\n}}
    \ex\Big[\exp\{-f_{\n,\bi}( T^{\n}_{1}\hat{\bB}_{\n,\bi})\}\Big]\Big|
    \to 0 \quad \text{as}\; \n \toi \, .
  \end{align*}
  Together with (\ref{eq:AI_blocks_cond}), this implies that, as $\n\toi$,
  \begin{multline}\label{eq:AI_original_blocks_aux3}
    \Delta_{\n}^{\delta}:=\Big|\ex\Big[\tprod_{\bi \in I_{\n}}
    \exp\left\{-f^{\delta}(\bi b_{\n} / {\n},
      T^{\n}_{1} {\bB}_{\n,\bi})\right\}\Big] -
    \tprod_{\bi\in I_{\n}}
    \ex\Big[ \exp{\left\{-f^{\delta}\left(\bi b_{\n} / {\n},
          T^{\n}_{1}{\bB}_{\n,\bi}\right)\right\}}\Big] \Big| \\
    =   \Big|\ex\Big[\tprod_{\bi \in I_{\n}}
    \exp{\left\{-f_{\n,\bi}(T^{\n}_{1} {\bB}_{\n,\bi})\right\}}\Big]
    - \tprod_{\bi \in I_{\n}}
    \ex\Big[\exp{\left\{-f_{\n,\bi}(T^{\n}_{1}
        {\bB}_{\n,\bi})\right\}}\Big]\Big|
    \to 0 .
  \end{multline}

  Using (\ref{eq:elem_ineq}) and the inequality
  $|e ^{-x}-e^{-y}|\leq |x-y|$ for $x,y\geq 0$, one gets
  \begin{align}
    \Big| \ex
    &\Big[\tprod_{\bi \in I_{\n}}
    \exp{\left\{-f(\bi b_{\n} / {\n}, T^{\n}_{1} [{\bB}_{\n,\bi}])\right\}}\Big] -
    \tprod_{\bi\in I_{\n}}
    \ex\Big[ \exp{\left\{-f\left(\bi b_{\n} / {\n},
          T^{\n}_{1}[{\bB}_{\n,\bi}]\right)\right\}}\Big] \Big|\notag \\
    &\le \Big|\ex\Big[\tprod_{\bi \in I_{\n}}
    \exp{\left\{-f(\bi b_{\n} / {\n}, T^{\n}_{1}
        [{\bB}_{\n,\bi}])\right\}}\Big] 
    - \ex\Big[\tprod_{\bi \in I_{\n}}
    \exp{\left\{-f^{\delta}(\bi b_{\n} / {\n},
        T^{\n}_{1} {\bB}_{\n,\bi})\right\}}\Big] \Big| \notag \\
    &+ \Big|\tprod_{\bi \in I_{\n}}
    \ex\Big[\exp{\left\{-f(\bi b_{\n} / {\n},
        T^{\n}_{1} [{\bB}_{\n,\bi}])\right\}}\Big] - \tprod_{\bi \in I_{\n}}
    \ex\Big[\exp{\left\{-f^{\delta}(\bi b_{\n} / {\n},
        T^{\n}_{1} {\bB}_{\n,\bi})\right\}}\Big] \Big| + \Delta_{\n}^{\delta}\notag \\
    &\le 2 \tsum_{\bi \in I_{\n}} \ex \Big|f(\bi b_{\n} / {\n},
    T^{\n}_{1} [{\bB}_{\n,\bi}]) - f^{\delta}(\bi b_{\n} / {\n},
    T^{\n}_{1} {\bB}_{\n,\bi})\Big| +  \Delta_{\n}^{\delta} \, .
    \label{eq:AI_original_blocks_aux4}
  \end{align}
  Since $M(\mu)\leq \epsilon$ implies $f(t,[\mu])=f^{\delta}(t,\mu)= 0$,
  for all $\bi\in I_{\n}$,
  \begin{align*}
    \ex \Big|f(\bi b_{\n} / {\n}, T^{\n}_{1} [{\bB}_{\n,\bi}])
    &- f^{\delta}(\bi b_{\n} / {\n}, T^{\n}_{1} {\bB}_{\n,\bi})\Big|\\
    &= \ex \Big|f(\bi b_{\n} / {\n}, T^{\n}_{1} [{\bB}_{\n,\bi}])
    - f^{\delta}(\bi b_{\n} / {\n}, T^{\n}_{1} {\bB}_{\n,\bi})
    \Big|\ind{M({\bB}_{\n,\bi})>a_{\n} \epsilon} \\
    &\leq c e^{-1/\delta} \,\pr(M({\bB}_{\n,\bi})>a_{\n} \epsilon)
    = c e^{-1/\delta} \,\pr(M(X_{[0,b_{\n}]^d})>a_{\n} \epsilon) \, ,
  \end{align*}
  where we used (\ref{eq:lipschitzness}) in the third step. Thus,
  the right-hand side in (\ref{eq:AI_original_blocks_aux4}) is bounded
  by
  \begin{align*}
    2c e^{-1/\delta} k_{\n}^d \;\pr(M(X_{[0,b_{\n}]^d})>a_{\n} \epsilon)
    + \Delta_{\n}^{\delta}  \leq 2 c e^{-1/\delta}
    k_{\n}^d b_{\n}^d \;\pr(\tX(0)>a_{\n} \epsilon) + \Delta_{\n}^{\delta},
  \end{align*}
  which by (\ref{eq:a_n_eps}) and (\ref{eq:AI_original_blocks_aux3})
  tends to $c e^{-1/\delta} \epsilon^{\alpha}$ as $\n\toi$. Since
  $\delta\in (0,\epsilon)$ is arbitrary, letting $\delta\to 0$ 
  finally yields (\ref{eq:AI_blocks_cond_for_original_blocks}).
\end{proof}

We are finally in position to prove \Cref{thm:main_result}.

\begin{proof}[Proof of \Cref{thm:main_result}]
  Since the family of Lipschitz continuous functions from
  \Cref{lem:I_blocks_cond_for_original_blocks} is convergence
  determining in the sense of
  \cite[Definition~2.1]{basrak:planinic:2020} (see
  \cite[Proposition~4.1]{basrak:planinic:2019}), the convergence of
  intensities (\ref{eq:intensity_conv_2}) and the asymptotic
  independence of blocks (\ref{eq:AI_blocks_cond_for_original_blocks})
  imply that, as $\n\toi$, $N_{\tau}$ converges in distribution
   to a Poisson point process $N$ on $\bN=[0,1]^d\times \tsN_{01}^*$
  whose intensity measure is $\Leb\times \nu$; this is
  \cite[Theorem~2.1]{basrak:planinic:2020} which is a consequence of
  the classical Grigelionis theorem,
  see~\cite[Corollary~4.25]{kallenberg:2017}. Standard
  transformation results for Poisson processes now imply that $N$ can
  be constructed as in (\ref{eq:pp_convergence}), and this finishes
  the proof of \Cref{thm:main_result}.
\end{proof}

\section{Acknowledgement}
\label{sec:acjknowledgements}

This work was supported by the Swiss Enlargement Contribution in the
framework of the Croatian--Swiss Research Programme (project number
IZHRZ0\_180549), and by the Croatian science foundation project IP-2022-10-2277.


\begin{thebibliography}{10}

\bibitem{aldous:1989}
D.~Aldous.
\newblock {\em Probability Approximations via the {P}oisson Clumping
  Heuristic}.
\newblock Springer-Verlag, New York, 1989.

\bibitem{basrak:planinic:2019}
B.~Basrak and H.~Planini{\'c}.
\newblock A note on vague convergence of measures.
\newblock {\em Statist. Probab. Lett.}, 153:180 -- 186, 2019.

\bibitem{basrak:planinic:2020}
B.~Basrak and H.~Planini\'{c}.
\newblock Compound {P}oisson approximation for regularly varying fields with
  application to sequence alignment.
\newblock {\em Bernoulli}, 27(2):1371--1408, 2021.

\bibitem{basrak18}
B.~Basrak, H.~Planini\'{c}, and P.~Soulier.
\newblock An invariance principle for sums and record times of regularly
  varying stationary sequences.
\newblock {\em Probab. Theory Related Fields}, 172(3-4):869--914, 2018.

\bibitem{basrak:segers:2009}
B.~Basrak and J.~Segers.
\newblock Regularly varying multivariate time series.
\newblock {\em Stochastic Process. Appl.}, 119(4):1055--1080, 2009.

\bibitem{billingsley:1968}
P.~Billingsley.
\newblock {\em Convergence of Probability Measures}.
\newblock New York, Wiley, 1968.

\bibitem{bobrowski21:_poiss_palm}
O.~Bobrowski, M.~Schulte, and D.~Yogeshwaran.
\newblock Poisson process approximation under stabilization and {P}alm
  coupling.
\newblock {\em Ann. H. Lebesgue}, 5:1489--1534, 2022.

\bibitem{otto-chen23}
N.~Chenavier and M.~Otto.
\newblock Compound {Poisson} process approximation under $\beta$-mixing and
  stabilization.
\newblock Technical report, arXiv math: 2310.15009, 2023.

\bibitem{chen-rob18}
N.~Chenavier and C.~Y. Robert.
\newblock Cluster size distributions of extreme values for the
  {P}oisson-{V}oronoi tessellation.
\newblock {\em Ann. Appl. Probab.}, 28(6):3291--3323, 2018.

\bibitem{chiu:2013}
S.~N. Chiu, D.~Stoyan, W.~S. Kendall, and J.~Mecke.
\newblock {\em Stochastic geometry and its applications}.
\newblock Wiley Series in Probability and Statistics. John Wiley \& Sons, Ltd.,
  Chichester, third edition, 2013.

\bibitem{daley:verejones:2003}
D.~J. Daley and D.~Vere-Jones.
\newblock {\em An Introduction to the Theory of Point Processes. {V}ol. {I}}.
\newblock Springer-Verlag, New York, second edition, 2003.

\bibitem{daley:verejones:2008}
D.~J. Daley and D.~Vere-Jones.
\newblock {\em An Introduction to the Theory of Point Processes. {V}ol. {II}}.
\newblock Springer, New York, second edition, 2008.

\bibitem{dombry:hashorva:soulier:2018}
C.~Dombry, E.~Hashorva, and P.~Soulier.
\newblock Tail measure and spectral tail process of regularly varying time
  series.
\newblock {\em Ann. Appl. Probab.}, 28(6):3884--3921, 2018.

\bibitem{eichelsbacher:2015}
P.~Eichelsbacher, M.~Rai\v{c}, and T.~Schreiber.
\newblock Moderate deviations for stabilizing functionals in geometric
  probability.
\newblock {\em Ann. Inst. H. Poincaré Probab. Statist.}, 51(1):89--128, 02
  2015.

\bibitem{janssen:2019}
A.~Jan{\ss}en.
\newblock Spectral tail processes and max-stable approximations of multivariate
  regularly varying time series.
\newblock {\em Stochastic Process. Appl.}, 129(6):1993 -- 2009, 2019.

\bibitem{kallenberg:2017}
O.~Kallenberg.
\newblock {\em Random Measures, Theory and Applications}.
\newblock Springer, Cham, 2017.

\bibitem{kulik:soulier:2020}
R.~Kulik and P.~Soulier.
\newblock {\em Heavy-Tailed Time Series}.
\newblock Springer, New York, 2020.

\bibitem{last:2021}
G.~Last.
\newblock Tail processes and tail measures: an approach via {P}alm calculus.
\newblock {\em Extremes}, 26(4):715--746, 2023.

\bibitem{morariu:2018}
M.~Morariu-Patrichi.
\newblock On the weak-hash metric for boundedly finite integer-valued measures.
\newblock {\em Bull. Aust. Math. Soc.}, 98(2):265--276, 2018.

\bibitem{otto20:_poiss_poiss}
M.~Otto.
\newblock Poisson approximation of {P}oisson-driven point processes and extreme
  values in stochastic geometry.
\newblock {\em Bernoulli}, 31(1):30--54, 2025.

\bibitem{owada18:_limit_betti}
T.~Owada.
\newblock Limit theorems for {B}etti numbers of extreme sample clouds with
  application to persistence barcodes.
\newblock {\em Ann. Appl. Probab.}, 28(5):2814--2854, 2018.

\bibitem{penrose:yukich:2001}
M.~D. Penrose and J.~E. Yukich.
\newblock Central limit theorems for some graphs in computational geometry.
\newblock {\em Ann. Appl. Probab.}, 11(4):1005--1041, 11 2001.

\bibitem{thesis}
H.~Planini{\'c}.
\newblock {\em Point Processes in the Analysis of Dependent Data}.
\newblock PhD thesis, University of Zagreb, 2019.

\bibitem{planinic:2021}
H.~Planini\'c.
\newblock Palm theory for extremes of stationary regularly varying time series
  and random fields.
\newblock {\em Extremes}, 26(1):45--82, 2023.

\bibitem{pratt:1960}
J.~W. Pratt.
\newblock On interchanging limits and integrals.
\newblock {\em Ann. Math. Statist.}, 31:74--77, 1960.

\bibitem{resnick:2007}
S.~I. Resnick.
\newblock {\em Heavy-Tail Phenomena}.
\newblock Springer, New York, 2007.

\bibitem{schuhmacher08}
D.~Schuhmacher and A.~Xia.
\newblock A new metric between distributions of point processes.
\newblock {\em Adv. in Appl. Probab.}, 40(3):651--672, 2008.

\end{thebibliography}

\end{document}